\PassOptionsToPackage{dvipsnames}{xcolor}
\documentclass{article}

\usepackage[T1]{fontenc}			% use also T2A option for cyrillic characters 

\usepackage{amsmath,amsthm,amsfonts}
\usepackage{mathtools}
\usepackage{latexsym}
\usepackage{tikz}
	\usetikzlibrary{decorations.pathmorphing,shapes,arrows,quotes,positioning,calc}
\usepackage{comment}
\usepackage{tikzit}				% String diagrams
\tikzstyle{morphism}=[fill=white, draw=black, shape=rectangle]
\tikzstyle{medium box}=[fill=white, draw=black, shape=rectangle, minimum width=0.7cm, minimum height=0.7cm]
\tikzstyle{large morphism}=[fill=white, draw=black, shape=rectangle, minimum width=1.7cm, minimum height=1cm]
\tikzstyle{bn}=[fill=black, draw=black, shape=circle, inner sep=1.5pt]
\tikzstyle{state}=[fill=white, draw=black, regular polygon, regular polygon sides=3, minimum width=0.8cm, shape border rotate=180, inner sep=0pt]
\tikzstyle{medium state}=[fill=white, draw=black, regular polygon, regular polygon sides=3, minimum width=1.3cm, inner sep=0pt, shape border rotate=180]
\tikzstyle{large state}=[fill=white, draw=black, regular polygon, regular polygon sides=3, minimum width=2.2cm, shape border rotate=180, inner sep=0pt]
\tikzstyle{wide state}=[fill=white, draw=black, shape=isosceles triangle, minimum width=0.8cm, shape border rotate=270, inner sep=1.4pt, minimum height=0.5cm, isosceles triangle apex angle=80]
\tikzstyle{wn}=[fill=white, draw=black, shape=circle, inner sep=1.5pt]
\tikzstyle{blue morphism}=[fill=white, draw={rgb,255: red,15; green,0; blue,150}, shape=rectangle, text={rgb,255: red,15; green,0; blue,150}, tikzit category=blue]
\tikzstyle{blue state}=[fill=white, draw={rgb,255: red,15; green,0; blue,150}, shape=circle, regular polygon, regular polygon sides=3, minimum width=0.8cm, shape border rotate=180, inner sep=0pt, text={rgb,255: red,15; green,0; blue,150}, tikzit category=blue]
\tikzstyle{blue node}=[fill={rgb,255: red,15; green,0; blue,150}, draw={rgb,255: red,15; green,0; blue,150}, shape=circle, tikzit category=blue, inner sep=1.5pt]
\tikzstyle{blue}=[text={rgb,255: red,15; green,0; blue,150}, tikzit draw={rgb,255: red,191; green,191; blue,191}, tikzit category=blue, tikzit fill=white, inner sep=0mm]
\tikzstyle{blue wide state}=[fill=white, draw={rgb,255: red,15; green,0; blue,150}, text={rgb,255: red,15; green,0; blue,150}, shape=isosceles triangle, minimum width=0.8cm, shape border rotate=270, inner sep=1.4pt, minimum height=0.5cm, isosceles triangle apex angle=80]
\tikzstyle{red node}=[fill={rgb,255: red,150; green,0; blue,2}, draw={rgb,255: red,150; green,0; blue,2}, shape=circle, inner sep=1.5pt]
\tikzstyle{Purple node}=[fill={rgb,255: red,120; green,0; blue,120}, draw={rgb,255: red,120; green,0; blue,120}, text={rgb,255: red,120; green,0; blue,120}, shape=circle, inner sep=1.5pt]
\tikzstyle{red}=[text={rgb,255: red,150; green,0; blue,2}, inner sep=0mm, tikzit fill=white, tikzit draw={rgb,255: red,191; green,191; blue,191}]
\tikzstyle{purple}=[text={rgb,255: red,150; green,0; blue,150}, inner sep=0mm, tikzit fill=white, tikzit draw={rgb,255: red,191; green,191; blue,191}]
\tikzstyle{white morphism}=[fill=white, draw=white, shape=rectangle, tikzit draw={rgb,255: red,139; green,139; blue,139}]
\tikzstyle{leak morphism}=[fill=white, draw={rgb,255: red,120; green,0; blue,85}, shape=rectangle, text={rgb,255: red,120; green,0; blue,85}, tikzit category=leak]
\tikzstyle{leak}=[text={rgb,255: red,120; green,0; blue,85}, inner sep=0mm, tikzit fill=white, tikzit draw={rgb,255: red,191; green,191; blue,191}, tikzit category=leak]
\tikzstyle{leak node}=[fill={rgb,255: red,120; green,0; blue,85}, draw={rgb,255: red,120; green,0; blue,85}, shape=circle, inner sep=1.5pt, tikzit category=leak]
\tikzstyle{curly brace}=[decorate, decoration={brace,amplitude=5pt}]

\tikzstyle{arrow}=[->]
\tikzstyle{dashed box}=[-, dashed]
\tikzstyle{blue arrow}=[-, draw={rgb,255: red,15; green,0; blue,150}, tikzit category=blue]
\tikzstyle{red arrow}=[-, draw={rgb,255: red,150; green,0; blue,2}, tikzit category=red]
\tikzstyle{purple line}=[draw={rgb,255: red,120; green,0; blue,120}, >=stealth, shorten <=2pt, shorten >=2pt, -]
\tikzstyle{protected purple line}=[draw={rgb,255: red,120; green,0; blue,120}, >=stealth, shorten <=2pt, shorten >=2pt, preaction={line width=1.8pt, white, draw}, -]
\tikzstyle{mapsto}=[{|->}]
\tikzstyle{double wire}=[-, double]
\tikzstyle{protected}=[-, preaction={line width=1.8pt,white,draw}]
\tikzstyle{leak arrow}=[-, line join=round, decorate, decoration={snake, segment length=4, amplitude=0.75, pre=curveto, post=curveto, pre length=1pt, post length=1pt}]
\tikzstyle{protected leak arrow}=[-, line join=round, decorate, decoration={snake, segment length=4, amplitude=0.75, pre=curveto, post=curveto, pre length=1pt, post length=1pt}, preaction={line width=1.8pt, white, draw}]
\tikzstyle{hollow arrow}=[-, very thin, white, preaction={line width=0.7pt,draw={rgb,255: red,120; green,0; blue,85}}, tikzit category=leak, tikzit draw={rgb,255: red,150; green,0; blue,120}]
\tikzstyle{protected hollow arrow}=[-, very thin, white, preaction={line width=0.7pt,draw={rgb,255: red,120; green,0; blue,85},preaction={line width=2.1pt,white,draw}}, tikzit category=leak, tikzit draw={rgb,255: red,150; green,0; blue,120}]

\usepackage{enumitem}
	\setlist[enumerate]{label=(\roman*)}  
	\setlist[enumerate,2]{label=(\alph*)}  
\usepackage[noadjust]{cite}
\usepackage{tikz-cd}
\usepackage{enumitem} 			% better enumerations
\usepackage{geometry}
\usepackage{bm}				% bold math
\usepackage{stmaryrd}			% semantic bracket \llbracket and \rrbracket
\usepackage{footmisc}			% provides \label and \ref for footnotes

\definecolor{myurlcolor}{rgb}{0,0,0.3}
\definecolor{mycitecolor}{rgb}{0,0.3,0}
\definecolor{myrefcolor}{rgb}{0.3,0,0}
\usepackage[pagebackref,draft=false]{hyperref}
\hypersetup{colorlinks,
	linkcolor=myrefcolor,
	citecolor=mycitecolor,
	urlcolor=myurlcolor}

\usepackage[noabbrev,capitalize]{cleveref}
\newcommand{\sref}[2]{\hyperref[#2]{#1~\ref{#2}}}	% A manual way to link theorems (and else) in a way so that links work in our figure 1

\newtheorem{theorem}{Theorem}[section]
\newtheorem{proposition}[theorem]{Proposition}
\newtheorem{lemma}[theorem]{Lemma}
\newtheorem{corollary}[theorem]{Corollary}
\newtheorem{definition}[theorem]{Definition}
\newtheorem{question}[theorem]{Question}

\theoremstyle{definition}
\newtheorem{example}[theorem]{Example}

\newtheorem{remark}[theorem]{Remark}

\newcommand{\N}{\mathbb{N}}
\newcommand{\Z}{\mathbb{Z}}
\newcommand{\R}{\mathbb{R}}

\newcommand{\ph}{\mathord{\rule[-0.05em]{0.6em}{0.05em}}}		% Argument placeholder
\newcommand{\newterm}[1]{\textbf{#1}}	% boldface font for newly introduced terms

\newcommand{\cat}[1]{{\mathsf{#1}}} 
\newcommand{\op}{\mathrm{op}}
\newcommand{\Setcat}{\mathsf{Set}}

\newcommand{\Kl}{\mathsf{Kl}}		% Kleisli category construction

\newcommand{\id}{\mathrm{id}} 		% identity

\tikzset{
	pullback/.style = { % pullback symbol in diagram
		minimum size=1.2ex,
		path picture={\draw[opacity=1,black,-,#1] (-0.5ex,-0.5ex) -- (0.5ex,-0.5ex) -- (0.5ex,0.5ex);%
			},
		}
	}

\DeclareMathOperator{\cop}{copy}
\DeclareMathOperator{\discard}{del}
\newcommand{\cC}{\mathsf{C}}			% Markov cat
\newcommand{\cD}{\mathsf{D}}			% cartesian cat
\newcommand{\dilations}{\mathsf{Dilations}}	% category of dilations
\renewcommand{\det}{\mathrm{det}}	% deterministic morphisms
\newcommand{\nc}{\mathrm{nc}}		% non-creative morphisms
\newcommand{\samp}{\mathsf{samp}}	% sampling map

\newcommand{\finstoch}{\mathsf{FinStoch}}
\newcommand{\borelmeas}{\mathsf{BorelMeas}}
\newcommand{\borelstoch}{\mathsf{BorelStoch}}

\newcommand{\stoch}{\mathsf{Stoch}}
\newcommand{\qbs}{\mathsf{Qbs}}
\newcommand{\qbstoch}{\mathsf{QBStoch}}

\newcommand{\as}[1]{% 								almost surely
	\def\relstate{#1}%
	\ifx\relstate\empty
		\text{a.s.}%
	\else
		{#1\text{-a.s.}}%
	\fi
}
\newcommand{\ase}[1]{=_{#1\text{-a.s.}}}			% almost surely equal
\newcommand{\dileq}[1]{=_{#1\text{-dil.}}}		% dilationally equal
\providecommand{\given}{}			% Just to make sure the \given command exists.
\newcommand{\SetSymbol}[1][]{%
	\nonscript\;\,#1\vert
	\allowbreak
	\nonscript\;\,
	\mathopen{}
}
\DeclarePairedDelimiterX{\Set}[1]{\{}{\}}{%
	\renewcommand{\given}{\SetSymbol[\delimsize]}
	#1
}
\makeatletter
\let\oldSet\Set
\def\Set{\@ifstar{\oldSet}{\oldSet*}}
\makeatother

\DeclarePairedDelimiterX{\Family}[1]{(}{)}{%
	\renewcommand{\given}{\SetSymbol[\delimsize]} 
	#1
}
\makeatletter
\let\oldFamily\Family
\def\Family{\@ifstar{\oldFamily}{\oldFamily*}}

\title{Dilations and Information Flow Axioms\\[2pt] in Categorical Probability}

\author{Tobias Fritz, Tom{\'a}{\v{s}} Gonda, Nicholas Gauguin Houghton-Larsen,\\[2pt] Antonio Lorenzin, Paolo Perrone and Dario Stein}

\begin{document}

\maketitle

\begin{abstract}
	We study the positivity and causality axioms for Markov categories as properties of dilations and information flow, and also develop variations thereof for arbitrary semicartesian monoidal categories.
	These help us show that being a positive Markov category is merely an additional \emph{property} of a symmetric monoidal category (rather than extra structure).
	We also characterize the positivity of representable Markov categories and prove that \emph{causality implies positivity}, but not conversely.
	Finally, we note that positivity fails for quasi-Borel spaces and interpret this failure as a privacy property of probabilistic name generation.
\end{abstract}

\tableofcontents

\section{Introduction}

	Markov categories are a categorical approach to the foundations of probability and statistics. 
	Recent developments of this framework have resulted in purely categorical proofs of various classical theorems, including theorems on sufficient statistics~\cite{fritz2019synthetic}, $0/1$-laws~\cite{fritzrischel2019zeroone}, comparison of statistical experiments~\cite{fritz2023representable}, the de Finetti theorem~\cite{fritz2021definetti,moss2022probability}, development of multinomial and hypergeometric distributions~\cite{jacobs2021multinomial}, ergodic systems~\cite{moss2022ergodic}, and the $d$-separation criterion for Bayesian networks~\cite{fritz2022dseparation}.
	The Markov categories framework has also found use in probabilistic programming theory~\cite{stein2021conditioning,stein2021structural} and cognitive science~\cite{smithe2021inference}.

	Many of these developments do not apply to arbitrary Markov categories, they require additional conditions, such as the \emph{existence of conditionals}, the \emph{causality axiom} or the \emph{positivity axiom} (see \Cref{sec:background} for more details). 
	The fact that these axioms hold in measure-theoretic probability is the only measure-theoretic input that is needed for developments like this. 
	The string-diagrammatic nature of these conditions also suggests that one can think of them as conditions on information flow, hence we propose to call them \newterm{information flow axioms}.

	The purpose of the present paper is to conduct a more detailed study of these axioms.
	Previously it was known that the existence of conditionals implies both the causality and the positivity axioms \cite[Proposition 11.34 and Lemma 11.24]{fritz2019synthetic}. 
	The converse is not true.
	For example, the Markov category $\cat{Stoch}$ of all measurable spaces and Markov kernels satisfies causality and positivity but does not have conditionals~\cite[Example 11.3]{fritz2019synthetic}.
	The relation between causality and positivity remained an open question.
	Our main result here is that \newterm{causality implies positivity, but not conversely} (\Cref{prop:causal_pos}).
	For both axioms, we also prove various reformulations which elucidate their meaning further.

	Besides adding further clarity to the intuition behind the axioms, these reformulations can also help in deciding whether a given Markov category satisfies them.
	As a case in point, we consider the Markov category $\qbstoch$, which is the Kleisli category of the probability monad on the category $\qbs$ of \newterm{quasi-Borel spaces}.
	We find that $\qbstoch$ violates positivity, and that it does so in an interesting way.
	As was recently discovered, $\qbs$ validates the privacy equation~\cite{sabok2021semantics}, which originally describes a phenomenon of fresh name generation in theoretical computer science \cite{stark:cmln}. 
	As we show in \Cref{prop:name_generation}, the privacy equation and positivity are incompatible.

	A secondary theme of this paper is the idea of developing categorical probability in terms of semicartesian monoidal categories only, which have a weaker structure than Markov categories.
	This is achieved using the concept of \newterm{dilations}. 
	A dilation of a Markov kernel is another Markov kernel with an additional output such that marginalizing over the latter recovers the original Markov kernel.
	In the case of kernels with trivial input, this amounts to an extension of the original probability space.
	While already many of our investigations of the positivity axiom for Markov categories are phrased in terms of dilations, the concept of dilation comes to shine in the purely semicartesian setting.
	There, we use them to define concepts that mimic those of almost sure equality and of deterministic morphisms in Markov categories.
	We also note that the structure of a positive Markov category can be recovered from its structure as a semicartesian category, and we provide a characterization of positive Markov categories in semicartesian and dilational terms.

\subsection{Summary of the paper}

	\begin{itemize}
		\item \Cref{sec:background} starts with the definitions of semicartesian categories and Markov categories, sketches the most important examples for this paper, and recalls the definition of dilation.
		
		\item \Cref{sec:positivity} presents a detailed study of the positivity axiom. 
			After a reformulation of positivity as \newterm{deterministic marginal independence} (DMI) in \Cref{def:dmp,thm:eqpos}, we derive a characterization of positivity for representable Markov categories in \Cref{rep_pos}.
			We then turn to the causality axiom and state its equivalence with \newterm{parametrized equality strengthening} in \Cref{def:pes,prop:pes_causality}.
			\Cref{prop:causal_pos} shows that causality implies positivity; an intricate counterexample for the converse is given by a Markov category with semiring-valued Markov kernels as morphisms for a carefully crafted semiring in \cref{prop:quantale_ex_positive_causal}.
			
		\item \Cref{sec:qbs} recalls the main features of quasi-Borel spaces before presenting the privacy equation as \Cref{thm:privacy}.
			\Cref{proposition:qbs_no_dmp} then uses our earlier reformulation of positivity to show that the Markov category of quasi-Borel spaces violates positivity. 
			\Cref{prop:name_generation} generalizes this to arbitrary categorical models of name generation by first observing that every such model defines a Markov category.
			
		\item \Cref{sec:semicart} treats aspects of categorical probability, and in particular the positivity theme, in purely semicartesian categories. 
			To this end, \Cref{def:as_eq} introduces \newterm{dilational equality}. %, which is a semicartesian variant of almost sure equality in Markov categories.
			It coincides with almost sure equality in Markov categories if and only if the said Markov category satisfies causality (\cref{prop:markov_dilation}).
			In \Cref{def:cat_dilations}, we associate a category of dilations to every morphism.
			Its initial objects are dubbed \newterm{initial dilations} (\Cref{def:initial_dilation}).
			This provides yet further characterizations of positivity for Markov categories as \Cref{prop:positivity_dilations,cor:pos_nc}.
			Finally, \Cref{cor:acp} characterizes positive Markov categories in terms of their structure as semicartesian categories alone, and \Cref{thm:acp} achieves the same for a slightly more general class of Markov categories.
	\end{itemize}
	
	\newlength{\chunit}
	\setlength{\chunit}{\dimexpr\numexpr 1*\textwidth/17 sp\relax}
	\newlength{\arshift}
	\setlength{\arshift}{10pt}
	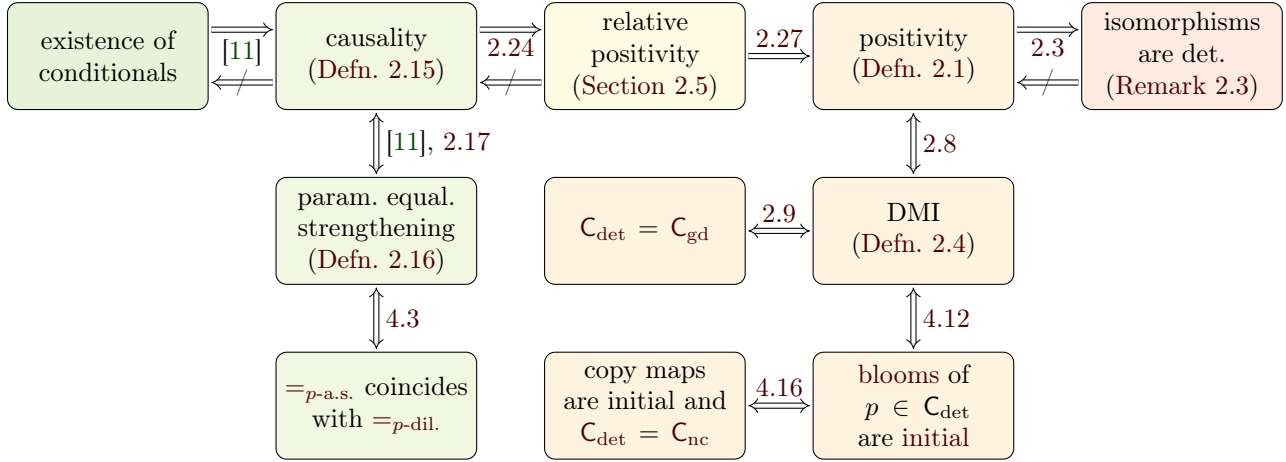
\begin{figure}
		\tikzset{%
			basic/.style = {rectangle, rounded corners, draw=black, minimum width=3\chunit, minimum height=1.6\chunit, text centered, text width=2.8\chunit, inner sep=2pt},
			axcond/.style = {basic,fill=LimeGreen!18!white},
			axcaus/.style = {basic,fill=SpringGreen!18!white},
			axrelpos/.style = {basic,fill=Yellow!13!white},
			axpos/.style = {basic,fill=Dandelion!15!white},
			axisosdet/.style = {basic,fill=RedOrange!11!white},
			implies/.style = {-{Implies},double,double equal sign distance},
			iff/.style = {{Implies}-{Implies},double,double equal sign distance},
		}
		\makebox[\textwidth][c]{
			\begin{tikzpicture}[node distance=\chunit,shorten > = 1pt, shorten < = 1pt]
				\node (cond) 		[axcond] 							{existence of conditionals};
				
				\node (caus)		[axcaus, right=of cond]				{causality\\ (\sref{Defn.}{defn:causal})};
				\node (pes)		[axcaus, below=of caus]		{param.\ equal.\ strengthening\\ (\sref{Defn.}{def:pes})};
				\node (dileq)		[axcaus, below=of pes]			{\hyperref[def:ase]{${\ase{p}}$} coincides with \hyperref[def:as_eq]{${\dileq{p}}$}};
				
				\node (relpos)		[axrelpos, right=of caus]				{relative positivity\\ (\sref{Section}{sec:relativized})};
				
				\node (pos)		[axpos, right=of relpos]				{positivity\\ (\sref{Defn.}{def:positivity})};
				\node (dmi)		[axpos, below=of pos]			{DMI\\ (\sref{Defn.}{def:dmp})};
				\node (gd)		[axpos, left=of dmi]			{$\hyperref[def:det]{\cC_{\rm det}} = \hyperlink{def:glob_det}{\cC_{\rm gd}}$};
				\node (bloom)		[axpos, below=of dmi]			{\hyperlink{def:bloom}{blooms} of $p \in \cC_{\rm det}$\!\!\\ are \hyperref[def:initial_dilation]{initial}};
				\node (copy)		[axpos, left=of bloom]			{copy maps are initial and $\hyperref[def:det]{\cC_{\rm det}} = \hyperref[def:noncreative]{\cC_{\rm nc}}$};
				
				\node (isosdet)		[axisosdet, right=of pos]				{isomorphisms are det.\ (\sref{Remark}{rem:isos_det})};

				\draw[implies] 		($(cond.east) + (0,\arshift)$) -- ($(caus.west) + (0,\arshift)$)	node[midway, below] {\textrm{\cite{fritz2019synthetic}}};
				\draw[implies] 		($(caus.west) + (0,-\arshift)$) -- ($(cond.east) + (0,-\arshift)$) 	node[midway] {/};
				\draw[implies] 		($(caus.east) + (0,\arshift)$) -- ($(relpos.west) + (0,\arshift)$)	node[midway, below] {\ref{prop:causal_pos}};
				\draw[implies] 		($(relpos.west) + (0,-\arshift)$) -- ($(caus.east) + (0,-\arshift)$) 	node[midway] {/};
				\draw[implies] 		(relpos) -- (pos)							node[midway, above] {\ref{rem:relpos_pos}};
				\draw[implies] 		($(pos.east) + (0,\arshift)$) -- ($(isosdet.west) + (0,\arshift)$)	node[midway, below] {\ref{rem:isos_det}};
				\draw[implies] 		($(isosdet.west) + (0,-\arshift)$) -- ($(pos.east) + (0,-\arshift)$) 	node[midway] {/};
				
				\draw[iff]			(caus) -- (pes)											node[midway, right] {\textrm{\cite{fritz2019synthetic}}, \ref{prop:pes_causality}};
				\draw[iff]			(pes) -- (dileq)											node[midway, right] {\ref{prop:markov_dilation}};
				
				\draw[iff]			(pos) -- (dmi)											node[midway, right] {\ref{thm:eqpos}};
				\draw[iff]			(dmi) -- (gd)											node[midway, above] {\ref{rem:glob_det}};
				\draw[iff]			(dmi) -- (bloom)											node[midway, right] {\ref{prop:positivity_dilations}};
				\draw[iff]			(bloom) -- (copy)										node[midway, above] {\ref{cor:pos_nc}};
			\end{tikzpicture}
		}
		\caption{Implications between various information flow axioms considered in this paper, with pointers to theorem numbers on the arrows.}
		\label{fig:implications}
	\end{figure}

	\sref{Figure}{fig:implications} summarizes various information flow axioms considered in this paper together with their relations.

	\subsubsection*{Prerequisites for reading}

	We assume that the reader has some basic familiarity with symmetric monoidal categories and string diagrams~\cite{baez2011rosetta,piedeleu2023strings}.
	Some prior exposure to Markov categories~\cite{fritz2019synthetic,chojacobs2019strings} will be helpful, but is not strictly necessary.
	On the probability theory side, basic knowledge of discrete probability theory suffices, as measure-theoretic probability does not play a central role in this paper.
	Somewhat of an exception is \Cref{sec:qbs};
	although we briefly recall the most relevant theoretical background, prior exposure to the theory of quasi-Borel spaces~\cite{heunen2017convenient} is helpful to understand it.

	\subsubsection*{Acknowledgements}

	We thank Arthur Parzgynat as well as two anonymous referees for a number of helpful comments on an earlier version.
	Tobias Fritz and Antonio Lorenzin acknowledge funding by the Austrian Science Fund (FWF) through project P 35992-N.
	Tom\'a\v{s} Gonda acknowledges the support of the START Prize Y 1261-N of the Austrian Science Fund (FWF).
	Paolo Perrone acknowledges funding from Sam Staton's grant BLaSt---a Better Language for Statistics of the European Research Council (ERC).

\subsection{Semicartesian categories, Markov categories, and dilations}
\label{sec:background}

	All categories that are of interest to us in this paper are symmetric monoidal categories with the following extra property, which for categorical probability implements the normalization of probability, and can be thought of as saying that there is a unique way to forget information.

	\begin{definition}
		\label{def:semicartesian}
		A symmetric monoidal category $\cC$ is \newterm{semicartesian} if the monoidal unit $I$ is terminal.	
	\end{definition}

	We usually abbreviate the lengthy phrase ``semicartesian symmetric monoidal category'' to ``semicartesian category'', leaving in particular the symmetry implicit since other cases will not be considered.
	Although the concept of semicartesian category is standard, we do not know where it was considered first. 

	Besides Markov categories, that we turn to next, interesting examples of semicartesian categories occur in quantum probability, where for example the category with finite-dimensional Hilbert spaces as objects and quantum channels as morphisms is a widely studied example (see \cite{houghtonlarsen2021dilations} and references therein).
	
	\paragraph{Markov categories.} Roughly speaking, a Markov category \cite{fritz2019synthetic} is a semicartesian category where all objects are commutative comonoids.
	Here is the precise definition. 
	
	\begin{definition}\label{defmarkovcat}
		A \newterm{Markov category} is a symmetric monoidal category $\cC$ where:
		\begin{itemize}
			\item Each object $X$ is equipped with ``copy'' and ``discard'' maps
				\begin{equation}
					\input{copy_del.tikz}
				\end{equation}
				satisfying the identities of a commutative comonoid:
				\begin{equation}
                			\input{comonoid.tikz} 
				\end{equation}
			\item The copy maps are compatible with the monoidal structure in the following way:
				\begin{equation}\label{copy_multiplicative}
					\input{multiplicativity.tikz}
				\end{equation}
			\item $\cC$ is semicartesian.
		\end{itemize}
	\end{definition}
	
	This definition is motivated by the fact that statements in probability theory often refer to the same variable multiple times, which explains why the copy morphisms are relevant.
	In a semicartesian category, we still have the discarding maps $\discard$, for which we use the same symbol in string diagrams.
	For further background on Markov categories we refer the reader to \cite{fritz2019synthetic}. For a history of the concept, see \cite[Introduction and Remark~2.2]{fritz2022free}
	
	A \newterm{state} in a Markov category, or more generally in a semicartesian category, is a morphism from the monoidal unit, i.e.\ in the form $m \colon I\to X$. We denote it by a triangle, 
		\begin{equation*}
			\begin{tikzpicture}
	\begin{pgfonlayer}{nodelayer}
		\node [style=none] (2) at (0, 1) {};
		\node [style=none] (4) at (0, 1.5) {$X$};
		\node [style=state] (5) at (0, 0) {$m$};
	\end{pgfonlayer}
	\begin{pgfonlayer}{edgelayer}
		\draw (5) to (2.center);
	\end{pgfonlayer}
\end{tikzpicture}

		\end{equation*}
    	States are the abstract categorical generalization of probability measures.
	
	\begin{example}[Probabilistic Markov categories]
		Here are some Markov categories of interest in probability theory and which are used in this work.
		\begin{itemize}
			\item The category $\cat{FinStoch}$ has as objects finite sets, and as morphisms stochastic matrices with their usual composition. We denote the entries of a matrix $m \colon X \to Y$ by $m(y|x)$, which we can interpret as a discrete transition probability.
			\item The category $\cat{Stoch}$ has as objects measurable sets, and as morphisms Markov kernels with their usual composition. Given such a kernel $k \colon X\to Y$, we denote its value on a point $x\in X$ and on a measurable subset $B\subseteq Y$ by $k(B|x)$, and again we can interpret it as the probability of obtaining an outcome in $B$ given input $x$. 
			\item The category $\cat{BorelStoch}$ is the full subcategory of $\cat{Stoch}$ whose objects are \emph{standard Borel spaces}, i.e.~measurable spaces whose $\sigma$-algebra can be written as the Borel $\sigma$-algebra of a complete separable metric space (often called a \emph{Polish space}).
			\item The category $\qbstoch$ has as objects quasi-Borel spaces and as morphisms kernels between them. We refer to \Cref{sec:qbs} and~\cite{heunen2017convenient,sabok2021semantics} for details.
		\end{itemize}
		In all these examples, states are probability measures.
		In the case of $\cat{FinStoch}$, they are discrete and finitely supported.
	\end{example}
	
	\begin{example}[Semiring-valued kernels]\label{ex:semiring_valuations}
		One can generalize the Markov category $\cat{FinStoch}$ to one in which transition ``probabilities'' are valued in an arbitrary commutative semiring $R$ instead of the semiring of non-negative real numbers.
		Furthermore, by requiring that, for each $x \in X$, the transition probability $m(y | x) \in R$ is non-zero for finitely many $y \in Y$ only, we can extend its objects to include all sets rather than merely the finite ones.
		In this way, we obtain a Markov category $\cat{Kl}(D_R)$, which is equivalently the Kleisli category of the $R$-distribution monad $D_R$ on the category of sets and functions.
		See \cite[Example 3.3]{fritz2023representable} or \cite[Section 5.1]{coumans2013scalars} for more details.
		For example, taking $R \coloneqq \R_+$ to be the nonnegative reals results in the category $\cat{Kl}(D_{\R_+})$, within which $\finstoch$ is the full subcategory on finite sets.
	\end{example}
	We return to Markov categories of semiring-valued kernels throughout the article.
	In particular, we identify properties of the semiring $R$ that characterize when they are positive and representable (\cref{prop:semiring_valuations_positivity}) as well as when they satisfy the causality axiom (\cref{prop:semiring_valuations_causality}).
	We use these to show that the Markov category $\cat{Kl}(D_R)$ is causal (\cref{prop:lattice_causal}) whenever $R$ is a bounded distributive lattice and to construct a Markov category that is positive but not causal (\cref{prop:quantale_ex_positive_causal}).

	\begin{example}[$\mathbb{R}$-valued kernels]\label{ex:negative_prob}
		For some counterexamples in this paper we also need the category $\finstoch_\pm$, which can be defined as the full subcategory of $\cat{Kl}(D_\R)$ on finite sets.
		More explicitly:
		\begin{itemize}
			\item Just as in $\finstoch$, objects are finite sets;
			\item Just as in $\finstoch$, a morphism $f \colon X\to Y$ is a $Y$-by-$X$ matrix whose columns sum to one,
				\begin{equation}
					\sum_{y\in Y} f(y|x) = 1.
				\end{equation}
				However, unlike in $\finstoch$, we do not require that the entries $f(y|x)$ are nonnegative.
			\item The copy and discard structures are the same as in $\finstoch$, which is embedded as a full subcategory.
		\end{itemize}
	\end{example}
	
	\begin{example}[Cartesian categories]
	 Every cartesian monoidal category is a Markov category with the copy morphisms given by the diagonal maps $X\to X\times X$. 
	 The categories $\cat{Set}$ and $\cat{FinSet}$ are examples. 
	\end{example}

	The main theme of this work is the \newterm{positivity axiom}, which we recall in \Cref{def:positivity}.
	It has played an important role in the proofs of theorems on sufficient statistics in Markov categories in \cite{fritz2019synthetic}, and variants of it have also been used in the quantum context \cite{parzygnat2020inverses}.

	In this work, we also make use of the following additional concepts for Markov categories, for which we also refer to earlier sources for more detail.

	\begin{definition}[{\cite[Definition 10.1]{fritz2019synthetic}}]\label{def:det}
		A morphism $f \colon A\to X$ in a Markov category is \newterm{deterministic} if it commutes with copying in the following way:
		\begin{equation}
			\input{deterministic.tikz}
		\end{equation}
		We denote by $\cC_\det$ the wide subcategory of $\cC$ consisting of deterministic morphisms.
	\end{definition}
	For example, a state in $\cat{Stoch}$ is deterministic if and only if it is a probability measure that assigns either $0$ or $1$ to each measurable subset.
	In $\cat{BorelStoch}$, these are exactly the Dirac delta measures.
	A cartesian monoidal category is exactly a Markov category where each morphism is deterministic.

	\begin{definition}[{\cite[Definition~11.5]{fritz2019synthetic}}]\label{def:conditionals}
		Given a morphism $f \colon A\to X\otimes Y$ in a Markov category, a \newterm{conditional} of $f$ \emph{given $X$} is a morphism $f_{|X} \colon X\otimes A\to Y$ such that
		\begin{equation}
			 \input{conditional.tikz}
		\end{equation}
		holds.
		We say that a Markov category \newterm{has conditionals} if every morphism admits a conditional.
	\end{definition}

	For example if $f$ is a state $I \to X \otimes Y$ in $\cat{Stoch}$, the conditional $f_{|X}$ recovers exactly the notion of \emph{regular conditional probability}~\cite[Example~11.3]{fritz2019synthetic}, and for general $f$, a conditional is the same thing with an additional measurable dependence on an additional parameter. 
	$\borelstoch$ has conditionals, but $\stoch$ does not since there are probability measures on product spaces without regular conditional probabilities~\cite{faden1985conditional}.
	We suggest \cite[Section 11]{fritz2019synthetic} for additional context.

	\begin{definition}[{\cite[Definition 13.1]{fritz2019synthetic}}]\label{def:ase}
		Given morphisms $f,g \colon A\to X$ in a Markov category, and a morphism $m \colon \Theta\to A$, we say that $f$ and $g$ are \newterm{$\bm{m}$-almost surely equal}, and write $f \ase{m} g$, if
		\begin{equation}
			 \input{ase.tikz}
		\end{equation}
	\end{definition}
	In $\cat{Stoch}$, especially when $m$ is a state, this recovers exactly almost sure equality with respect to a measure~\cite[Example~13.3]{fritz2019synthetic}.

	\begin{definition}[{\cite[Definition 3.10]{fritz2023representable}}]
		A Markov category $\cC$ is \newterm{representable} if the inclusion functor $\cC_\det\to\cC$ has a right adjoint $P \colon \cC \to \cC_\det$. In this case, we write
		\begin{equation}
			\label{eq:sampling}
			\samp_X \colon PX \to X
		\end{equation}
		for the counit of the adjunction, and $f^\sharp \colon A \to PX$ for the deterministic counterpart of a morphism $f \colon A \to X$.
	\end{definition}

	We denote the right adjoint, and also the induced monad on $\cC_\det$, by $P$ in order to evoke the association with a probability monad.
	For example, $\borelstoch$ is representable, since a Markov kernel $A \to X$ is the same thing as a measurable map $A \to PX$, where $PX$ is the measurable space of probability measures on $X$.
	The counit~\eqref{eq:sampling} is the Markov kernel given by $\samp_X(S|\mu) = \mu(S)$ for every measurable $S \in \Sigma_X$, which we interpret as the kernel which outputs a random sample from every distribution $\mu$.

	A Markov category $\cC$ is representable if and only if it is the Kleisli category of an affine commutative monad on $\cC_\det$~\cite[Section~3.2]{fritz2023representable}. 
	Because of this, a representable Markov category can be thought of as adding probabilistic effects (via the monad) to a cartesian category.
	For example, $\borelstoch$ is the Kleisli category of the Giry monad on $\borelmeas = \borelstoch_\det$.

	\medskip

	Besides the above, a very important notion for this work is that of \emph{dilations}.
	While this concept has been established for some time, in particular for quantum states and processes \cite{chiribella2010purification,chiribella2014dilations,selby2021reconstructing}, dilations have not been systematically studied or applied in the general setting of semicartesian categories before the third author's PhD thesis \cite{houghtonlarsen2021dilations}.
	
	\begin{definition}\label{def:dilation}
		Let $\cD$ be a semicartesian category (i.e.\ $\cD$ need not be a Markov category).
		Let $p \colon A \to X$ be any morphism in $\cD$. 
		A \newterm{dilation} of $p$ is any morphism $\pi \colon A \to X \otimes E$ for some $E \in \cC$ satisfying
		\begin{equation}\label{eq:dilation}
			\input{dilation.tikz}
		\end{equation}
		We call $E$ the \newterm{environment} of the dilation.
	\end{definition}
	
	Intuitively, a dilation $\pi$ describes a process which coincides with $p$ while potentially leaking information to an ``environment'' $E$.
	Conversely, $p$ is obtained from $\pi$ by ignoring the leaked information.
	Dilations have been studied extensively in \cite{houghtonlarsen2021dilations} in order to give an abstract account of the so-called \emph{self-testing} of quantum instruments.
	There, various concepts relating to the structure of dilations were introduced (such as completeness, universality, localisability, purifiability), and these give a flavour of the various types of properties which dilations may or may not enjoy in a given category.
	Dilations have also been utilized in \cite{selby2021reconstructing}\footnotemark{} to formulate a categorical purification axiom as part of a characterization of quantum theory among other theories of physical processes.
	We use dilations for a similar purpose in \Cref{sec:initial_dilations,sec:non-creative}, namely to characterize positivity in Markov categories, as well as distinguishing positive Markov categories among other semicartesian categories.
	We study them in detail in \Cref{sec:semicart}.
	\footnotetext{The dilations used in \cite{houghtonlarsen2021dilations} and in \cite{selby2021reconstructing} are generally \emph{two-sided}, which means that the $g$ in \Cref{def:dilation} may also carry an additional input not shared with $f$. Note, however, that the definitions of two-sided dilations in \cite{houghtonlarsen2021dilations} and \cite{selby2021reconstructing} are not the same. One-sided dilations as per \cref{def:dilation} are sufficient for our purposes. \label{fn:no_two_sided}}%

	We end this background section by relating dilations to convex combinations, which adds some further intuition to the concept of dilation.
	In \cite[Section~3.1]{moss2022ergodic} it is shown that for categories of Markov kernels, a convex decomposition of a state can be expressed by means of categorical composition. 
	We now sketch how convex combinations can also be expressed in terms of dilations, in a way that generalizes to all morphisms (not just states).
	The basic idea is that a dilation $\pi \colon A\to X\otimes E$ of $p \colon A \to X$ can express explicitly the different terms, indexed by $E$, which appear in a convex decomposition of $p$. 
	
	\begin{example}[Convex combinations as dilations]\label{exconv}
		In $\cat{FinStoch}$, consider the state $p \colon I\to X$ over $X = \{x,y,z\}$ given by
		\begin{align}
			p(x) &= 1/6, & p(y) &= 3/6, & p(z) &= 2/6
		\end{align}
		We would like to express the distribution $p$ as a convex combination 
		\begin{equation}\label{eq:convex}
			p = \frac{1}{3} \,p_1 + \frac{2}{3} \,p_2
		\end{equation}
		where 
		\begin{align}
			& \begin{split} p_1(x) &= 1/2,	 \\ p_2(x) &= 0, \end{split} & &
			& \begin{split} p_1(y) &= 1/2,	 \\ p_2(y) &= 1/2, \end{split} & &
			& \begin{split} p_1(z) &= 0,		 \\ p_2(z) &= 1/2. \end{split} & &
		\end{align}
		To this end, we can use the environment $E=\{1,2\}$ and the dilation $\pi \colon I\to X\times E$ given by 
		\begin{align}
			& \begin{split} \pi(x,1) &= \frac{1}{3}\,p_1(x),		 \\ \pi(x,2) &= \frac{2}{3}\,p_2(x), \end{split} & 
			& \begin{split} \pi(y,1), &= \frac{1}{3}\,p_1(x),	 \\ \pi(y,2), &= \frac{2}{3}\,p_2(x), \end{split} & 
			 \begin{split} \pi(z,1) &= \frac{1}{3}\,p_1(x),		 \\ \pi(z,2) &= \frac{2}{3}\,p_2(x). \end{split} & 
		\end{align}
		In this way, we can view $\pi$ as equivalent to the convex decomposition from \eqref{eq:convex}. 
		The object $E$ indexes the ``components'' of $p$ via the conditional distribution $\pi_{|E} \colon E \to X$.
		Furthermore, the coefficients appearing in \Cref{eq:convex} correspond to the marginal distribution of $\pi$ on $E$. 
	\end{example}

	In general, if we have a morphism $p \colon A \to X$ for general $A$, viewed as a parametrized state, the same procedure gives a decomposition where both the coefficients and the components in \Cref{eq:convex} are allowed to depend on a parameter ranging over $A$. 
	While this implements finite convex combinations in $\finstoch$, in measure-theoretic probability such as $\borelstoch$, one obtains infinitary convex decompositions in the form of integrals.

	Let us now establish a link between dilations and the decompositions of states of \cite{moss2022ergodic}.
	Let $\cC$ be a Markov category, and let $p \colon I\to X$ be a state. 
	In \cite{moss2022ergodic}, decomposing $p$ means expressing it as a composition 
	\begin{equation}
		\label{eq:comp_decomp}
		\begin{tikzcd}
			I \ar{r}{m} & E \ar{r}{k} & X
		\end{tikzcd}
	\end{equation}
	for some object $E$, which indexes the components (and $m$ plays the role of the ``coefficients'', see \cite[Section~3.1]{moss2022ergodic} for more). 
	Such a decomposition always gives a dilation given by
	\begin{equation}\label{dilationfromcomp}
		\input{conv_dil.tikz}
	\end{equation}
	and this dilation represents $p$ as a convex combination in a way that is equivalent to~\eqref{eq:comp_decomp}.
	In fact whenever conditionals exist, the two kinds of decomposition are equivalent.
	
	\begin{proposition} 
		Let $\cC$ be a Markov category with conditionals, and let $p \colon I\to X$ be a state. 
		For every object $E$, the construction of \Cref{dilationfromcomp} establishes a bijective correspondence between:
		\begin{itemize}
			\item dilations of $p$ with environment $E$, and
			\item decompositions of $p$ via $E$ up to almost sure equality, i.e.~pairs $(m,[k])$ where 
				\begin{itemize}
					\item $m \colon I\to E$ is a state;
					\item $[k]$ is an equivalence class of morphisms $k \colon E\to X$ modulo \as{$m$} equality;
					\item $k \circ m = p$ for each $k\in [k]$.
				\end{itemize}
		\end{itemize}
	\end{proposition}
	
	The proof can be seen as an abstraction of \Cref{exconv}.
	
	\begin{proof}
		\Cref{dilationfromcomp} shows how to map from decompositions to dilations.
		The surjectivity of this map is immediate by the existence of conditionals.
		For injectivity, note that $m$ can be recovered as the marginal on $E$; then the claim holds by definition of $m$-almost sure equality.
	\end{proof}
	
\section{Conditions for positivity of Markov categories}
\label{sec:positivity}

	The positivity axiom for Markov categories, as introduced in~\cite{fritz2019synthetic}, formalizes the idea that any (potentially random) intermediate outcome of a deterministic process is independent of the output given the input.
	The goal of this section is to present a number of reformulations of the positivity axiom for Markov categories.
	These underline its significance and shed further light on its intuitive meaning.
	Along the way, we develop a number of related notions that may be of independent interest.

	\subsection{The positivity axiom}

		Here, we recall the definition of a positive Markov category.
		Readers familiar with this material may proceed directly to \cref{sec:dmp}.
		
		\begin{definition}
			\label{def:positivity}
			A Markov category $\cC$ is \newterm{positive} if whenever $f \colon X\to Y$ and $g \colon Y\to Z$ are such that $g \circ f$ is deterministic, then we have
			\begin{equation}\label{positivity}
				\input{positivity.tikz}
			\end{equation}
		\end{definition}
		
		The positivity property was introduced under this name in \cite{fritz2019synthetic} because the proof that it holds in $\stoch$ relies importantly on the nonnegativity of probabilities.
		It was also observed there that positivity follows from the existence of conditionals.
		Moreover, positivity fails in $\finstoch_\pm$ \cite[Example 11.27]{fritz2019synthetic}, which provides a nice way to see that $\finstoch_\pm$ does not have conditionals.
		
		\begin{remark}\label{rem:positive}
			\begin{enumerate}
				\item The intuition behind the positivity axiom is that if a composite computation $gf$ is deterministic, then it is possible to calculate the intermediate result (the output of $f$) independently of the result of the entire computation.
				\item The stronger notion of strict positivity\footnote{\label{fn:strict_pos}This looks like an unfortunate choice of terminology in hindsight, and we are in favour of changing it in the future.} relativizes the positivity axiom with respect to almost sure equality \cite[Definition 13.16]{fritz2019synthetic}. 
				This property is relevant for proving properties of sufficient statistics, namely versions of Fisher--Neyman factorization theorem \cite[Theorem 14.5]{fritz2019synthetic} and of Basu's theorem \cite[Theorem 15.8]{fritz2019synthetic}.
				We briefly consider strict positivity, and refer to it as \emph{relative positivity}, in \cref{sec:relativized}.
		
				\item \label{item:quantumpositive} Parzygnat \cite[Section~4]{parzygnat2020inverses} has considered a quantum analogue of the positivity axiom, but for subcategories of a (quantum) Markov category. 
				It is indeed related to notions of positivity in the quantum setting.
				In particular, the category of linear unital maps between finite-dimensional C*-algebras includes Schwarz positive maps (and thus also completely positive maps) as a subcategory that satisfies the relevant positivity axiom. 
				Nevertheless, the subcategory of all positive linear maps does not satisfy the categorical notion of positivity \cite[Example 4.9]{parzygnat2020inverses}.
			\end{enumerate}
		\end{remark}

		\begin{remark}
			\label{rem:isos_det}
			As a rather weak information flow axiom, one may consider the property that every isomorphism is deterministic.
			This is not the case in every Markov category~\cite[Remark~10.10]{fritz2019synthetic};
			For example in $\finstoch_\pm$ defined in \cref{ex:negative_prob}, every invertible matrix is an isomorphism, but it is only the permutation matrices among them which are also deterministic.

			The property that every isomorphism is deterministic is a simple consequence of positivity~\cite[Remark~11.28]{fritz2019synthetic}.
			Indeed if $f$ is an isomorphism, then taking $g = f^{-1}$ in \Cref{positivity} and composing with $f$ on the right output shows that $f$ is deterministic.

			Conversely, the condition ``isomorphisms are deterministic'' does not imply positivity.
			Finding a Markov category that witnesses this distinction is, however, not trivial.
			That is why we introduce a Markov category that has not yet been considered elsewhere as far as we know.
			
			Let $\cC$ be the symmetric monoidal category where:
			\begin{itemize}
				\item Objects are commutative monoids, written in additive notation with binary operation $+$ and neutral element $0$;
				\item A morphism $f$ in $\cC(X,Y)$ is given by a submonoid $S_f \subseteq Y$ called \textit{support} and a monoid homomorphism $\tilde{f} \colon S_f \to X$ (in particular, $\tilde{f}$ can be viewed as a partial function $Y \to X$);
				\item A sequential composite $g \circ f$ is given by the function $\tilde{f} \circ \tilde{g}$ with support
					\begin{equation}\label{eq:det_supp_cond}
						S_{g \circ f} \coloneqq \Set{ z \in S_g  \given  \tilde{g}(z) \in S_f } = S_g \cap \tilde{g}^{-1} \left( S_f \right);
					\end{equation} 
				\item The tensor product of objects is the cartesian product of monoids, and the tensor product $f \otimes g$ of morphisms $f \colon X \to Y$ and $g \colon Z \to W$ is given by $\tilde{f} \times \tilde{g} \colon S_f \times S_g \to X \times Z$.
			\end{itemize}
			The monoidal unit is the trivial monoid $I = \{0\}$.
			This becomes a Markov category when we identify
			\begin{itemize}
				\item discarding $\discard_X \colon X \to I$ as given by the unique monoid homomorphism $I \to X$ mapping $0$ to the neutral element of $X$, and 
				\item copying $\cop_X \colon X \to X \otimes X$ as given by addition with full support: $(x_1, x_2) \mapsto x_1 + x_2$.
			\end{itemize}
			In this Markov category, a morphism $h \colon X \to Y$ is deterministic if and only if its support is divisor-closed \cite{garcia2019divisor}, i.e.\ if we have
			\begin{equation}
				y_1 + y_2 \in S_h \quad \implies \quad y_1 \in S_h \text{ and } y_2 \in S_h
			\end{equation}
			for all $y_1, y_2 \in Y$.
			Therefore, isomorphisms in $\cC$, which necessarily have full support, are deterministic.
			
			Let us now show that $\cC$ defined as above is \emph{not} a positive Markov category.
			To demonstrate this, consider the morphisms $f \colon I \to \N$ and $g \colon \N \to \N$ given by
			\begin{align}
				S_f &= \N \setminus \{1\},  &  \tilde{f}(n) &= n, \\
				S_g &= \N,  &  \tilde{g}(n) &= 2n.
			\end{align} 
			Then the composite $g \circ f$ is deterministic, because its support is $\N$.
			This follows from \cref{eq:det_supp_cond} since the image of $\tilde{g}$ is within the support of $f$.
			
			However, \cref{positivity} does not hold for these morphisms.
			The support of its left-hand side contains $(1,1)$ because $1 + \tilde{g}(1) = 3$ belongs to $S_f$.
			On the other hand, the support of its right-hand side is given by $\N \times S_f$, which does not contain $(1,1)$.
			Therefore, the Markov category defined above is not positive, even though all of its isomorphisms are deterministic.
%			For example, consider the wide subcategory of $\finstoch_\pm$, a morphism of which is
%			\begin{itemize}
%				\item an arbitrary stochastic matrix (i.e.\ a morphism of $\finstoch$), or
%				\item any morphism of $\finstoch_\pm$ whose rank is $1$.
%			\end{itemize}
%			Since this class of morphisms is closed under composition, tensor product and contains all the structure morphisms of $\finstoch_\pm$, it follows that it is a Markov category in its own right.
%			To see that every isomorphism is deterministic, note that a stochastic matrix of rank $1$ is not invertible unless its domain and codomain are both singletons, in which case its only entry must be $1$.
%			To see that positivity fails, consider e.g.
%			\begin{equation}
%				f = \left( \begin{matrix}
%					1 \\
%					1 \\
%					-1
%				\end{matrix} \right),
%				\qquad
%				g = \left( \begin{matrix}
%					1 & 0 & 0 \\
%					0 & 1 & 1
%				\end{matrix} \right),
%			\end{equation}
%			for which $g \circ f$ is deterministic, but \Cref{positivity} does not hold.
		\end{remark}
	
	\subsection{Deterministic marginal independence is equivalent to positivity}\label{sec:dmp}
	
		In probability theory, it is an obvious fact that a deterministic random variable is independent of any other random variable.
		This fact has previously made a brief appearance in the Markov categories framework in~\cite[Proposition~12.14]{fritz2019synthetic}.
		Here is the general definition.
		
		\begin{definition}\label{def:dmp}
			A Markov category $\cC$ satisfies \newterm{deterministic marginal independence} (\newterm{DMI}) if for every deterministic morphism $p \colon A \to X$, every dilation $\pi \colon A \to X \otimes E$ of $p$ displays the conditional independence of $X$ and $E$ given $A$, i.e.\
			\begin{equation}\label{eq:cond_ind}
				\input{cond_ind.tikz}
			\end{equation}
		\end{definition}
		
		Equivalently, DMI says that every morphism $A \to X \otimes E$ that has a deterministic marginal (say, on $X$) must display conditional independence of $X$ and $E$ given $A$.
		
		Note that the term ``deterministic marginal independence'' is intended to be understood as ``(deterministic marginal) independence'', not as ``deterministic (marginal independence)''.
		
		In words, deterministic marginal independence states that a deterministic output of a process cannot be correlated with another output.
		The following example already illustrates that this property is also related to the nonnegativity of probabilities.
		
		\begin{example}[DMI for stochastic matrices]\label{ex:deterministic marginal property_FinStoch}
			$\finstoch$ satisfies deterministic marginal independence.
			For instance, consider trivial input $A = I$, so that $\pi$ is a joint distribution of $X$ and $E$.
			Deterministic states in $\finstoch$ are point distributions, so that we have $p = \delta_{x_0}$ for some $x_0 \in X$. % since these are the only deterministic states.
			The assumption that $\pi$ dilates $p$ means that for $x \in X$,
			\begin{equation}\label{eq:deterministic marginal property_FinStoch}
				\sum_{e \in E} \pi(x,e) = \begin{cases} 1 & \text{if } x = x_0, \\ 0 & \text{otherwise} \end{cases}
			\end{equation}
			holds.
			By the nonnegativity of probabilities, this implies that $\pi(x,e)$ vanishes for all $e$ whenever $x \neq x_0$.
			Consequently, the other marginal of $\pi$ is given by 
			\begin{equation}
				\pi_E(e) = \pi(x_0,e),
			\end{equation}
			and the whole joint distribution can be written as
			\begin{equation}
				\pi(x,e) = \delta_{x_0}(x) \cdot \pi(x_0, e) = p(x) \cdot \pi_E(e)
			\end{equation}
			which is precisely the desired \Cref{eq:cond_ind}.
			For more general morphisms, the same holds, where now both $p$ and $\pi$ depend on an additional parameter.
		\end{example}
		
		Another closely related notion is the following one.
		
		\begin{definition}%[Determinism in $X$]
			\label{det_in_X}
			A morphism $q \colon A\to X\otimes E$ is \newterm{deterministic in $X$} if and only if it satisfies
			\begin{equation}\label{eq:det_in_X}
				\input{det_in_X.tikz}
			\end{equation} 
		\end{definition}
		
		Discarding the output $E$ shows that if $q$ is deterministic in $X$, then its marginal $q_X$ is deterministic. 
		However, the converse does generally not hold, as the following example shows. 
		
		\begin{example}[Deterministic marginals with negative probabilities]
			\label{pm_fails_dmp}
			In $\finstoch_\pm$, a joint distribution $q \colon 1\to X\otimes E$ is deterministic in $X$ if and only if for all outcomes $x_1, x_2 \in X$ and $e \in E$, we have
			\begin{equation}\label{detXstoch}
				\delta_{x_1, x_2} \cdot q(x_2, e)  = q_X(x_1) \cdot q(x_2, e) .
			\end{equation}
			In $\finstoch$, this property is equivalent to the marginal $q_X$ being deterministic; this is an instance of \Cref{thm:eqpos} below.
			In $\finstoch_\pm$, instead, there are joint distributions $q$ which are not deterministic in $X$ although their marginal $q_X$ is deterministic. 
			For example, for $X = \{x,y\}$ and $E = \{a,b\}$, taking the signed distribution with joint probabilities given by
			\begin{equation}
				\begin{array}{c|cc}
					q & a & b \\
					\hline
					x & 1/2 & 1/2 \\
					y & 1/2 & -1/2
				\end{array}
			\end{equation}
			has marginal $q_X$ equal to $\delta_x$, which is deterministic. However, setting $x_1=x_2=y$ and $e=a$ in \eqref{detXstoch} results in
			\begin{equation}
				1\cdot 1/2 \ne 0 \cdot 1/2 .
			\end{equation}
			The culprit is that the marginal $q_X$ gives mass zero to $y$, but the joint probability $q$ gives nonzero mass to the point $(y,a)$. This is possible because in this category we are allowing negative probabilities, and some mass cancels out when we form the marginal.
		\end{example}
		
		\begin{theorem}\label{thm:eqpos}
			For a Markov category $\cC$, the following are equivalent:
			\begin{enumerate}[label=(\roman*)]
				\item\label{condpos} $\cC$ is positive.
				\item\label{conddmp} $\cC$ satisfies deterministic marginal independence.
				\item\label{conddetX} For all $q \colon A \to X \otimes E$,
					\[
						q \textrm{ is deterministic in } X \quad \Longleftrightarrow \quad q_X \textrm{ is deterministic}.
					\]
			\end{enumerate}
		\end{theorem}
		\begin{proof} \hfill 
			\begin{itemize}
				\item[$\ref{condpos}\Rightarrow\ref{conddmp}$:]
					This was proven as~\cite[Proposition~12.14]{fritz2019synthetic}, and we recall the argument here for completeness.
					If $p$ is deterministic, then in the defining \Cref{positivity} we take a dilation $\pi$ thereof in place of $f$ and the morphism
					\begin{equation}
						\input{margx.tikz}
					\end{equation}
					in place of $g$.
					Then the composite $g \circ f$ is equal to $p$, which is deterministic by assumption. 
					\Cref{positivity} now reads
					\begin{equation}
						\input{posp.tikz}
					\end{equation}
					We then get the desired conditional independence by marginalizing over the leftmost output and swapping the other two,
					\begin{equation}
						\input{cond_ind.tikz}
					\end{equation}
					implicitly using the commutativity of $\cop$.
					
				\item[$\ref{conddmp}\Rightarrow\ref{condpos}$:]
					Consider $f \colon A\to X$ and $g \colon X\to Y$ with a deterministic composite $g \circ f$. 
					Then the morphism
					\begin{equation}
						\input{defp.tikz} 
					\end{equation}
					is a dilation of its $Y$-marginal $g \circ f$, which is deterministic by assumption. 
					Therefore deterministic marginal independence applies and gives
					\begin{equation}
						\input{dmpp.tikz}
					\end{equation}
					as was to be shown.
		
				\item[$\ref{conddmp}\Rightarrow\ref{conddetX}$:]
					We already noted that determinism in $X$ of $q$ always implies that its $X$-marginal $q_X$ is deterministic, so we focus on the backward implication, assuming that $\cC$ satisfies DMI.
		
					Consider a morphism $q\colon A \to X \otimes E$ with deterministic marginal $q_X$.
					By DMI, $q$ displays the conditional independence of $X$ and $E$ given $A$.
					Using both of these properties entails
					\begin{equation}\label{eq:proof_dmp_to_detX}
						\input{proof_dmp_to_detX.tikz}
					\end{equation}
					so that $q$ is indeed deterministic in $X$.
					The first and last equations hold because of the conditional independence; the second one uses the fact that $q_X$ is deterministic; and the third one holds by associativity of copying.
		
				\item[$\ref{conddetX}\Rightarrow\ref{conddmp}$:]
					Let $\pi \colon A\to X\otimes E$ be a dilation of a deterministic morphism $p \colon A \to X$.
					Then $\pi$ is necessarily deterministic in $X$ by the assumed Property \ref{conddetX}. 
					Marginalizing the middle output in \Cref{eq:det_in_X} gives
					\begin{equation}
						\input{cdh.tikz}
					\end{equation}
					which is the desired conditional independence.
					\qedhere
			\end{itemize}
		\end{proof}
		\begin{remark}
			\label{rem:glob_det}
			Let us call a morphism $p \colon A \to X$ \hypertarget{def:glob_det}{\newterm{globally deterministic}} if every dilation of $p$ is deterministic in $X$.
			We can then express Property \ref{conddetX} as saying 
			\begin{equation}
				\cC_{\rm det} = \cC_{\rm gd}
			\end{equation}
			where $\cC_{\rm gd}$ denotes the class of globally deterministic morphisms in $\cC$.
		\end{remark}
	
		As a consequence of \cref{thm:eqpos}, \Cref{pm_fails_dmp} corresponds to the failure of positivity in $\finstoch_\pm$ as noticed in \cite[Example 11.27]{fritz2019synthetic}.

		We can also use \Cref{thm:eqpos} to establish the conditions under which Markov categories of semiring-valued stochastic matrices from \Cref{ex:semiring_valuations} are positive.
		First, let us characterize the deterministic morphisms therein.
		To that end, recall that a commutative semiring $R$ is \newterm{entire} if $0 \neq 1$ and if $R$ has no zero divisors in the sense that 
		\begin{equation}
			rs = 0  \quad \implies \quad  r = 0 \;\text{ or }\; s = 0
		\end{equation}
		holds.
	
		\begin{lemma}\label{lem:pure_is_deterministic}
			Let $R$ be an entire commutative semiring.
			A morphism $f \colon A \to X$ in $\cat{Kl}(D_R)$ is deterministic if and only if it can be expressed as 
			\begin{equation}\label{eq:pure_is_deterministic}
				f(x | a) = \delta_{f^{\flat}(a)}(x)
			\end{equation}
			for a function $f^{\flat} \in \cat{Set}(A,X)$, where $\delta_{x'} \in D_R(X)$ is the delta distribution given by
			\begin{equation}
				x \mapsto
					\begin{cases}
						1 &  \text{if }  x = x',   \\
						0 &  \text{if }  x \neq x' .
		      		\end{cases}
			\end{equation}
			Moreover, $f^\flat$ is then uniquely determined by $f$.
		\end{lemma}
		
		In other words, deterministic morphisms of $\cat{Kl}(D_R)$ coincide with its pure morphisms in the sense of \cite[Definition 2.5]{moss2022probability}.
		However, they do not necessarily coincide with pure or dilationally pure morphism in the sense of \cite{selby2021reconstructing} and \cite{houghtonlarsen2021dilations} respectively.
		Note that the function $f^{\flat}$ is unique by $1 \neq 0$ in $R$.

		\begin{proof}
			Since every morphism of that form is clearly deterministic, it suffices to show the forward implication.
			Let $f$ be deterministic, meaning that 
			\begin{equation}\label{eq:deterministic_distribution}
				f(x | a) \, f(x' | a) = f(x | a) \, \delta_{x'}(x)
			\end{equation}
			holds for all $a \in A$ and all $x,x' \in X$.
			For every $a$ there is an $x$ with $f(x | a) \neq 0$.
			Since $R$ has no zero divisors by assumption, \Cref{eq:deterministic_distribution} which takes the form $f(x | a) \, f(x' | a) = 0$ then implies $f(x' | a) = 0$ for every $x'$ distinct from $x$.
			Normalization then forces $f(x | a) = 1$, so that $f$ satisfies \Cref{eq:pure_is_deterministic}.

			The uniqueness is clear by $0 \neq 1$ in $R$; equivalently, the canonical Kleisli functor $\cat{Set} \to \cat{Kl}(D_R)$ is faithful.
		\end{proof}
		
		While the proof of \Cref{lem:pure_is_deterministic} is instructive, it is worth noting that this statement also follows from \cite[Propositions 3.4 and 3.6]{fritz2023representable}, which in turn also implies that the Markov category $\cat{Kl}(D_R)$ is representable.\footnote{Because of $\Kl(D_R)_\det = \Setcat$, the relevant right adjoint to the inclusion $\Kl(D_R)_\det \hookrightarrow \Kl(D_R)$ is simply the usual Kleisli adjoint of the inclusion $\Setcat \hookrightarrow \Kl(D_R)$.}
		
		\begin{definition}\label{def:zero_sum_free}
			A semiring $R$ is \newterm{zerosumfree} if it satisfies
			\begin{equation}\label{eq:zero_sum_free}
				r + s = 0 \quad \implies \quad r = s = 0
			\end{equation}
			for all $r,s \in R$.
		\end{definition}
		
		\begin{proposition}\label{prop:semiring_valuations_positivity}
			Let $R$ be an entire commutative semiring.
			Then the Markov category $\cat{Kl}(D_R)$ from \Cref{ex:semiring_valuations} is positive if and only if $R$ is zerosumfree.
		\end{proposition}

		This generalizes the fact that $\cat{FinStoch}_\pm$ is not positive,
		and can be taken as further motivation for the term ``positivity''.

		\begin{proof}
			We use the characterization of positive Markov categories as those satisfying deterministic marginal independence.
			
			First, assume that $R$ is zerosumfree and consider a dilation $\pi \colon A \to X \otimes E$ of a deterministic $p \colon A \to X$.
			By \Cref{lem:pure_is_deterministic}, we have
			\begin{equation}
				\sum_{e \in E} \pi(x,e|a) = p(x | a) = \delta_{p^{\flat}(a)}(x).
			\end{equation}
			Therefore by zerosumfreeness we have $\pi(x,e|a) = 0$ for every $e \in E$ and every $x \in X$ distinct from $p^\flat(a) \in X$.
			This means that 
			\begin{equation}
				\pi(x,e|a) = \delta_{p^{\flat}(a)}(x) \, \tilde{p}(e|a)
			\end{equation}
			holds, where $\tilde{p} \colon A \to E$ is the $E$-marginal of $\pi$.
			But this is precisely the statement of deterministic marginal independence.
			
			Conversely, assume that $R$ is not zerosumfree, i.e.\ that there exist $r$ and $s$ such that $r + s = 0$ and $r \neq 0$.
			Consider a morphism $\pi \colon I \to X \otimes E$ with $X = \{x,x'\}$ and $E = \{e, e'\}$, %\{ {\uparrow}, {\downarrow}\}$ 
			given by the joint distribution $\pi \in D_R(X \times E)$ with
			\begin{align}
				\begin{aligned}
					\pi(x, e) &= 1, \\
					\pi(x, e') &= 0,
				\end{aligned} &&
				\begin{aligned}
					\pi(x', e) &= r, \\
					\pi(x', e') &= s.
				\end{aligned}
			\end{align}
			Then the marginal $p \coloneqq \pi_X$ is deterministic.
			However, $\pi$ is not equal to the product of its marginals, since the latter is instead given by
			\begin{align}
				\begin{aligned}
					\pi_X (x) \, \pi_E(e) &= 1 + r, \\
					\pi_X (x) \, \pi_E(e') &= s,
				\end{aligned} &&
				\begin{aligned}
					\pi_X (x') \, \pi_E(e) &= 0, \\
					\pi_X (x') \, \pi_E(e') &= 0.
				\end{aligned}
			\end{align}
			This means that $\cat{Kl}(D_R)$ does not satisfy deterministic marginal independence and thus it is not positive.
		\end{proof}
	
	\subsection{Positivity of representable Markov categories}
	
		We will show next that the following existing notion can be used to detect the positivity of a representable Markov category, in terms of the associated commutative monad on the cartesian monoidal category $\cC_{\rm det}$ of deterministic morphisms.
		
		\begin{definition}[{Jacobs \cite[Definition 1]{stronglyaffine}}]
			A strong monad $(P,\mu,\delta)$ with strength $s$ on a cartesian monoidal category $\cD$ is \newterm{strongly affine} if for all objects $X$ and $Y$ of $\cD$, the diagram
			\begin{equation}\label{pbpos}
				\begin{tikzcd}
					X\times PY \ar{r}{s} \ar{d}{\pi_1} & P(X\times Y) \ar{d}{P(\pi_1)} \\
					X \ar{r}{\delta} & PX
				\end{tikzcd}
			\end{equation}
			is a pullback in $\cD$, where $\pi_1$ is the projection map.
		\end{definition} 
		
		Taking $X = Y = 1$ shows that a strongly affine monad is in particular affine (satisfies $P1 \cong 1$).
		Therefore the adverb ``strongly'' cleverly refers to both a strengthening of affineness and to the fact that the condition involves the strength $s$.
		
		\begin{proposition}
			\label{rep_pos}
			Let $\cC$ be a representable Markov category with affine commutative monad $P$ on $\cC_\det$, so that $\cC = \Kl(P)$. Then $\cC$ is positive if and only if $P$ is strongly affine.
		\end{proposition}
		
		\begin{proof}			
			By \Cref{thm:eqpos}, it suffices to show that $\cC$ satisfies deterministic marginal independence if and only if $P$ is strongly affine.
		
			In a representable Markov category, a morphism $p \colon A \to X$ is deterministic if and only if it satisfies $p^\sharp = \delta p$~\cite[Proposition~3.12]{fritz2023representable}.
			For the remainder of the proof, we work in $\cC_\det$ only.
			Then by the previous statement, a given $q \colon A \to X \times Y$ has deterministic first marginal $q_X$ if and only if there exists $f \colon A \to X$ (namely $q_X$) such that the diagram
			\begin{equation}
				\begin{tikzcd}
					A \ar[bend left]{drr}{q^\sharp} \ar[swap,bend right]{ddr}{f} \\
					& X\times PY \ar{r}[swap]{s} \ar{d}{\pi_1} & P(X\times Y) \ar{d}{P(\pi_1)} \\
					& X \ar{r}[swap]{\delta} & PX
				\end{tikzcd}
			\end{equation}
			commutes.
			Now, if the square \eqref{pbpos} is a pullback, then $q$ factors (uniquely) through the strength $s$, i.e.~there exists a unique morphism $u \colon X\to Y\times PZ$ making the diagram 
			\begin{equation}
				\begin{tikzcd}
					A \ar[bend left]{drr}{q^\sharp} \ar[swap,bend right]{ddr}{q_X} \ar[dotted, "u" description]{dr} \\
					& X\times PY \ar{r}[swap]{s} \ar{d}{\pi_1} & P(X\times Y) \ar{d}{P(\pi_1)} \\
					& X \ar{r}[swap]{\delta} & PX
				\end{tikzcd}
			\end{equation}
			commute.
			By the universal property of the product $X\times PY$ in $\cC_\det$, the map $u$ is determined by its components. 
			Its $X$-component must be equal to $q_X$ and its $PY$-component must be equal to the marginal $q_Y^\sharp$ in order for the diagram to commute.
			Indeed, since the diagram
			\begin{equation}
				\begin{tikzcd}
					X\times PY \ar{dr}[swap]{\pi_2} \ar{r}{s} & P(X\times Y) \ar{d}{P(\pi_2)} \\
					& PY
				\end{tikzcd}
			\end{equation}
			commutes, we have that $\pi_2\circ u = P(\pi_2)\circ s \circ u = P(\pi_2)\circ q^\sharp = q_Y^\sharp$.
			Recall now that the strength $s$ is given by the following composition,
			\begin{equation}
				\begin{tikzcd}
					X\times PY \ar{r}{\delta\times\id} & PX \times PY \ar{r}{\nabla} & P(X\times Y),
				\end{tikzcd}
			\end{equation}
			where $\nabla$ denotes the lax symmetric monoidal structure morphism of $P$.
			Therefore, if \eqref{pbpos} is a pullback, then we get
			\begin{equation}
				q^\sharp = s\circ u = \nabla\circ (\delta\times\id) \circ (f,q_Y^\sharp) = \nabla \circ ( q_X^\sharp, q_Y^\sharp).
			\end{equation}
			Sampling on both sides produces the desired factorization of \Cref{eq:cond_ind}.
			Since $q$ was arbitrary, it follows that $\cC$ has deterministic marginal independence.
		
			The converse implication follows by the same line of argument upon noting that the above reasoning covers every instance of the universal property of the pullback in $\cC_\det$.
		\end{proof}
	
	\subsection{Causality and positivity}
		
		We now turn to another important information flow axiom: causality \cite[Definition~11.31]{fritz2019synthetic}.
		
		\begin{definition}[{\cite[Definition~11.31]{fritz2019synthetic}}]
			\label{defn:causal}
			A Markov category $\cC$ is \newterm{causal} if whenever $f \colon A\to W$, $g \colon W\to X$ and $h_1,h_2 \colon X\to Y$ satisfy
			\begin{equation}\label{eq:causal1}
				\input{causal1.tikz}
			\end{equation}
			then we also have the stronger equation
			\begin{equation}\label{eq:causal2}
				\input{causal2.tikz}
			\end{equation} 
		\end{definition}

		Intuitively, the axiom states that if a choice between $h_1$ and $h_2$ in the ``future'' of $g$ does not affect anything that happens from there on, then this choice cannot affect anything that happened in the ``past'' of $g$ either.
		
		To show that causality is a stronger property than the positivity axiom, it is helpful to have an alternative formulation thereof.
		The following definition elaborates on~\cite[Remark~11.36]{fritz2019synthetic}.
		
		\begin{definition}%[Parameterized Equality Strengthening]
			\label{def:pes}
			A Markov category $\cC$ has \newterm{parametrized equality strengthening} if for any $h_1,h_2 \colon X \to Y$ and any $p \colon A \to X$, 
			\begin{equation}\label{eq:eq_str1}
				\input{eq_str1.tikz}
			\end{equation}
			implies that for every dilation $\pi$ of $p$ with environment $E$, we have
			\begin{equation}\label{eq:eq_str2}
				\input{eq_str2.tikz}
			\end{equation}
		\end{definition}
		\Cref{def:pes} extends the notion of equality strengthening given by Cho and Jacobs \mbox{\cite[p.\ 19]{chojacobs2019strings}}, who considered the special case in which $p$ has trivial input, meaning that $A = I$.
		
		\begin{proposition}[{\cite[Remark~11.36]{fritz2019synthetic}}]\label{prop:pes_causality}
			For a Markov category $\cC$, the following are equivalent:
			\begin{enumerate}[label=(\roman*)]
				\item \label{it:causality} $\cC$ is causal.
				\item \label{it:PES} $\cC$ has parametrized equality strengthening.
			\end{enumerate} 
		\end{proposition}
		\begin{proof} The proof amounts to reinterpreting the terms in the respective equalities.
			\begin{itemize}
				\item[$\ref{it:causality}\Rightarrow\ref{it:PES}$:] Consider $h_1$, $h_2$, and $p$ satisfying \cref{eq:eq_str1} and an arbitrary dilation $\pi \colon A \to X \otimes E$ of $p$.
					If we define
					\begin{equation}\label{eq:pes_to_caus}
						\input{pes_to_caus.tikz}
					\end{equation}
					then \cref{eq:eq_str1} coincides with \eqref{eq:causal1} and applying causality gives 
					\begin{equation}\label{eq:pes_to_caus_2}
						\input{pes_to_caus_2.tikz}
					\end{equation}
					which, upon marginalizing the two $X$ outputs, gives the required \cref{eq:eq_str2}.
					
				\item[$\ref{it:PES}\Rightarrow\ref{it:causality}$:] Conversely, the morphism 
					\begin{equation}\label{eq:causal_to_pes}
						\input{causal_to_pes.tikz}
					\end{equation}
					is a particular dilation of $g \circ f \colon A \to X$ with environment $X \otimes W$, so that applying parametrized equality strengthening to \cref{eq:causal1} gives \cref{eq:causal2}. \qedhere
			\end{itemize}
		\end{proof}
		
		Let us use \Cref{prop:pes_causality} to characterize causality for Markov categories of semiring-valued stochastic matrices from \Cref{ex:semiring_valuations}.
		\begin{definition}
			An element $r$ of a semiring $R$ is said to have a \newterm{complement} $\overline{r} \in R$ if we have
			\begin{equation}
				r + \overline{r} = 1 .
			\end{equation}
		\end{definition}

		\begin{remark}
			\label{lattice_case}
			A complement need not be unique if it exists.
			For example, $R$ may satisfy $x + x = x$ and $x^2 = x$ for all $x \in R$, in which case $R$ is equivalently a bounded distributive lattice with join $+$ and meet $\cdot$, and with $0$ as bottom and $1$ as top element.
			In this case, $1$ is a complement for (and is complemented by) every other element.\footnote{Unfortunately our notion of complement does not coincide with the usual notion of lattice complement in this case, since we only require $r \lor \overline{r} = 1$ but not $r \land \overline{r} = 0$.}
		\end{remark}

		\begin{proposition}\label{prop:semiring_valuations_causality}
			Let $R$ be a commutative semiring.
			The Markov category $\cat{Kl}(D_R)$ is causal if and only if, for all $s,t,v,w \in R$ such that $s,t$ and $v+w$ have complements in $R$, we have the following implication:
			\begin{equation}\label{eq:weakly_zero_sum_free}
				s (v + w)  = t (v + w)   \quad \implies \quad  s v = t v  \;\text{ and }\; s w = t w
			\end{equation}
		\end{proposition}
		\begin{proof}
			First, let us show that for a semiring satisfying Implication \eqref{eq:weakly_zero_sum_free}, the corresponding Markov category $\cat{Kl}(D_R)$ has parametrized equality strengthening. 
			Writing out \Cref{eq:eq_str1} in components gives
			\begin{equation}\label{eq:eq_str1_comp}
				h_1(y | x) \, p(x | a) = h_2(y | x) \, p(x | a)
			\end{equation}
			which, for any dilation $\pi$ of $p$, reads
			\begin{equation}\label{eq:pas_sum}
				h_1(y | x) \left[ \sum_{e \in E} \pi(x,e|a) \right] = h_2(y | x) \left[ \sum_{e \in E} \pi(x,e|a) \right] .
			\end{equation}
			For any choice of $a \in A$, $e \in E$, $x \in X$, and $y \in Y$, let us define the following elements of $R$,
			\begin{align}
				\begin{aligned}
					s &\coloneqq h_1(y | x), \\[2pt]
					\overline{s} &\coloneqq \sum_{y' \neq y} h_1(y' | x),
				\end{aligned} &&
				\begin{aligned}
					t &\coloneqq h_2(y | x), \\[2pt]
					\overline{t} &\coloneqq \sum_{y' \neq y} h_2(y' | x),
				\end{aligned} &&
				\begin{aligned}
					v &\coloneqq \pi(x,e | a), \\[2pt]
					w &\coloneqq \sum_{e' \neq e} \pi(x,e' | a).
				\end{aligned} 
			\end{align}
			Then, by normalization of the respective morphisms, $s$ and $t$ have $\overline{s}$ and $\overline{t}$ as complements while $v+w$, being equal to $p(x | a)$, has a complement too.
			Moreover, \Cref{eq:pas_sum} takes the form of the antecedent of Implication \eqref{eq:weakly_zero_sum_free}.
			Using this implication then gives $s v = t v$, which reads
			\begin{equation}\label{eq:eq_str2_comp}
				h_1(y | x) \, \pi(x,e | a) = h_2(y | x) \, \pi(x,e | a).
			\end{equation}
			This is the componentwise form of \Cref{eq:eq_str2} that we aimed to show.
			
			Conversely, assume that $\cat{Kl}(D_R)$ has parametrized equality strengthening.
			Let $A = I$ be the singleton and let $X$, $Y$ and $E$ each have cardinality two.
			We choose morphisms $p$, $\pi$, $h_1$ and $h_2$ with types as above to be given by
			\begin{align}
				\begin{aligned}
					p &\coloneqq \begin{pmatrix}  z \\ \overline{z} \end{pmatrix} \\[6pt]
					\pi &\coloneqq \begin{pmatrix} v & w \\ \overline{z} & 0 \end{pmatrix}
				\end{aligned} &&
				\begin{aligned}
					h_1 &\coloneqq \begin{pmatrix}  s & 1 \\ \overline{s} & 0 \end{pmatrix} \\[6pt]
					h_2 &\coloneqq \begin{pmatrix}  t & 1 \\ \overline{t} & 0 \end{pmatrix} .
				\end{aligned} 
			\end{align}
			where $z \coloneqq v + w$.
			In the above matrix notation, $\pi$ ranges over $X$ in its rows and over $E$ in its columns, while $h_1$ and $h_2$ denote stochastic matrices $X \to Y$ with $X$ in columns and $Y$ in rows as usual.
			With these choices, \Cref{eq:eq_str1_comp} reads
			\begin{equation}
				\label{eq:sztz}
				\begin{pmatrix}  s z & \overline{s} z \\ \overline{z} & 0 \end{pmatrix} = \begin{pmatrix}  t z & \overline{t} z \\ \overline{z} & 0 \end{pmatrix},
			\end{equation}
			as an equality of joint distributions in $D_R(X \times Y)$ with $X$ on the rows and $Y$ on the columns.
			This equation follows from the assumed antecedent of Implication \eqref{eq:weakly_zero_sum_free}. %, the only non-trivial case being the bottom left component, which 
			Applying parametrized equality strengthening to get \Cref{eq:eq_str2_comp}, we then obtain the requisite equations $sv = tv$ and $sw = tw$ by choosing the elements of $X$ and $Y$ that correspond to the upper left corner in \Cref{eq:sztz}.
		\end{proof}
		
		\begin{remark}[Causality implies positivity for semiring-valued kernels]
			For entire commutative semirings, \cref{prop:semiring_valuations_causality,prop:semiring_valuations_positivity} foreshadow \cref{prop:causal_pos}, which says that causality implies positivity for arbitrary Markov categories.
			Indeed, any commutative semiring satisfying Implication \eqref{eq:weakly_zero_sum_free} is also zerosumfree.
			To see this, let $s = 1$ and $t = 0$, both of which have complements.
			Since $v + w = 0$ implies that $v+w$ also has a complement, applying \eqref{eq:weakly_zero_sum_free} gives $v = 0$ and $w = 0$.
			Thus, $R$ is then zerosumfree and $\cat{Kl}(D_R)$ is a positive Markov category by \cref{prop:semiring_valuations_positivity}.
		\end{remark}
		\begin{remark}[Multiplicative cancellativity implies causality]
			Another consequence of \cref{prop:semiring_valuations_causality} is that if $R$ is a zerosumfree commutative semiring in which multiplication by nonzero elements can be cancelled, then the Markov category $\cat{Kl}(D_R)$ is causal.
			Indeed, if $v + w = 0$, then by zerosumfreeness we conclude $v = w = 0$, so that Implication \eqref{eq:weakly_zero_sum_free} is satisfied.
			Otherwise, the equation $s (v + w)  = t (v + w)$ gives us $s = t$ by cancelling $v + w$, and thus the implication also holds in this case.
			An interesting example where this applies would be the \emph{tropical semiring} $R = [-\infty, +\infty)$ with $\max$ as addition and $+$ as multiplication.
		\end{remark}
		As we show next, the conditions in \cref{prop:semiring_valuations_causality} are also satisfied when $R$ is a bounded distributive lattice (as considered in \Cref{lattice_case}).
		Later on, this will rule such semirings out as potential counterexamples showing that positivity does not imply causality, and we will need to consider more complicated semirings instead (\Cref{prop:quantale_ex_positive_causal}).

		\begin{proposition}\label{prop:lattice_causal}
			Let $R$ be a bounded, distributive lattice. %, and thus a commutative semiring with addition and multiplication given by suprema and infima respectively.
			Then the Kleisli category $\cat{Kl}(D_R)$ is a causal Markov category.
		\end{proposition}
		\begin{proof}
			Let us denote the underlying lattice ordering by $\geq$.
			We will show that for arbitrary elements $s,t,v,w$ of $R$, Implication \eqref{eq:weakly_zero_sum_free} holds, from which causality follows by \Cref{prop:semiring_valuations_causality}.
			By the antecedent of \eqref{eq:weakly_zero_sum_free} and the fact that in a lattice we always have $a+b \geq a \geq ab$ and $a+b \geq b \geq ab$, we can infer the order relations depicted in the following Hasse diagram, where $z \coloneqq v + w$:
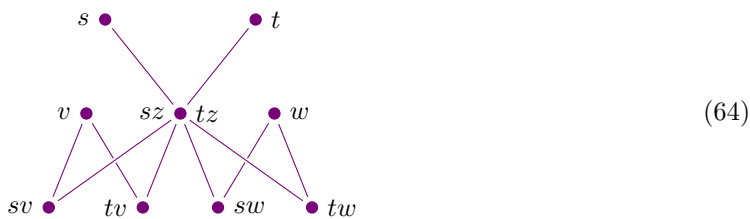
			\begin{equation}\label{eq:lattice_causal}
				\input{lattice_causal.tikz}
			\end{equation}
			Using transitivity of $\geq$, we extract relations
			\begin{align}
				s &\geq tv,  &  v &\geq tv,  &  t &\geq sv,  &  v &\geq sv.
			\end{align}
			Since $sv$ is the greatest lower bound of $\{s,v\}$, the first two imply $sv \geq tv$,
			and by similar reasoning the latter two imply $tv \geq sv$. 
			Since $\geq$ is antisymmetric, we finally get $sv = tv$.
			Analogously, one can obtain the other equation $sw = tw$.
			By \Cref{prop:semiring_valuations_causality}, $\cat{Kl}(D_R)$ is thus a causal Markov category.
		\end{proof}
		
		Returning to the general theory, we now prove a new and surprising implication from causality to positivity.
		
		\begin{theorem}\label{prop:causal_pos}
			If a Markov category $\cC$ is causal, then $\cC$ is positive. The converse is false.
		\end{theorem}
		\begin{proof}
			Let $f \colon A \to W$ and $g \colon W \to X$ be such that $g \circ f$ is deterministic. %, i.e.\ we have
			Define $p$ to be the morphism
			\begin{equation}\label{eq:caus_to_pos_1}
				\input{caus_to_pos_1.tikz}
			\end{equation}
			where the two forms of $p$ are equal by the assumption that $g \circ f$ is deterministic.
			Then there is a dilation $\pi$ of $p$, with environment $W$, given by
			\begin{equation}\label{eq:caus_to_pos_2}
				\input{caus_to_pos_2.tikz}
			\end{equation}
			Note that the following equation
			\begin{equation}\label{eq:caus_to_pos_4}
				\input{caus_to_pos_4.tikz}
			\end{equation}
			holds by the associativity of copying, where we identify $h_1$ as ${\discard_X} \otimes {\id_X}$ and $h_2$ as ${\id_X} \otimes {\discard_X}$.
			
			Since causality is equivalent to parametrized equality strengthening (\cref{prop:pes_causality}), we can apply the latter to \cref{eq:caus_to_pos_4}.
			Replacing $p$ (with its outputs copied) by the dilation $\pi$ from \eqref{eq:caus_to_pos_2} yields
			\begin{equation}\label{eq:caus_to_pos_5}
				\input{caus_to_pos_5_2.tikz}
			\end{equation}
			which is the desired positivity equation up to swapping the outputs (see \cref{def:positivity}).

			The fact that positivity does not imply causality in general is shown by constructing a Markov category that is positive but not causal, which we do next in \Cref{prop:quantale_ex_positive_causal}.
		\end{proof}
		
		In particular, we take inspiration from \Cref{prop:semiring_valuations_positivity,prop:semiring_valuations_causality} and look for an entire commutative semiring that is zerosumfree, but does not satisfy Implication \eqref{eq:weakly_zero_sum_free}.
		By \Cref{prop:lattice_causal}, we know that such a semiring cannot be a distributive lattice.
		Rather, let $\mathcal{I}(\Z[2i])$ be the commutative quantale\footnote{Recall that a \newterm{quantale} is a semiring where addition is given by the join of a complete join-semilattice, and such that multiplication distributes over arbitrary joins. The lattice of ideals of a commutative ring is a commutative quantale under ideal multiplication~\cite[Chapter~4]{rosenthal1990quantales}.} of ideals of the commutative ring
		\begin{equation}
			\Z[2i] = \Z \oplus 2i \Z = \Set*[\big]{  m \oplus 2i k \given m,k \in \Z },
		\end{equation}
		which is the subring of the Gaussian integers whose imaginary part is even.
		The addition and multiplication in $\mathcal{I}(\Z[2i])$ are the ideal addition and multiplication, respectively, with units given by the null ideal $\{0\}$ the whole ring $\Z[2i]$, respectively.
		
		\begin{proposition}\label{prop:quantale_ex_positive_causal}
			Let $R$ be the semiring of ideals $\mathcal{I}(\Z[2i])$ defined above.
			Then:
			\begin{enumerate}
				\item $R$ is entire and zerosumfree.\footnote{As the proof will show, this holds with any integral domain in place of $\Z[2i]$.}
				\item The Markov category $\cat{Kl}(D_R)$ is representable and positive.
				\item The Markov category $\cat{Kl}(D_R)$ is not causal.
			\end{enumerate} 
		\end{proposition}
		\begin{proof}
			\begin{enumerate}
				\item $R$ is non-trivial as we have $\{0\} \neq \Z[2i]$.
					To show that it is entire, we thus need to prove that is has no zero divisors.
					Consider two non-zero ideals $I,J \subseteq \Z[2i]$ with $IJ = \{0\}$.
					The latter is equivalent to $\alpha \beta = 0$ for all $\alpha \in I$ and $\beta \in J$.
					Since $\Z[2i]$ itself is entire, this implies $I = J = \{0\}$.

					To show that $R$ is zerosumfree, consider two ideals $I$ and $J$ that sum to zero, i.e.\ $I + J = \{0\}$.
					Since the sum of ideals contains each of them, the desired $I = J = \{0\}$ is immediate.
				\item Representability follows by \cite[Proposition 3.6]{fritz2023representable} and the previous item.
					Positivity follows by \Cref{prop:semiring_valuations_positivity} and the previous item.
				\item Let us show that Implication \eqref{eq:weakly_zero_sum_free} fails here.
					To this end, we choose ideals
					\begin{align}
						s &= v \coloneqq (2, 4i),  &  t &= w \coloneqq (4, 2i),
					\end{align}
					where $(m,2ik)$ denotes the set $m \mathbb{Z} \oplus 2ik \mathbb{Z} \subseteq \Z[2i]$.
					Note that $s$ and $t$ are the principal ideals generated by $2$ and $2i$ respectively.
					We use the $(m,2ik)$ notation to highlight that these are distinct ideals in $\Z[2i]$, which would not be the case for the Gaussian integers $\Z[i]$.
					Since every element of $R$ has a complement given by $\Z[2i]$ itself, $s$, $t$, and $(v+w)$ do as well.
					We can now compute
					\begin{align}
						s^2 &= t^2 = (4, 8i),  &  st &= (8,4i),
					\end{align}
					so that we have 
					\begin{equation}
						s (v + w) = s^2 + st = st + t^2 = t (v + w),
					\end{equation}
					but nevertheless
					\begin{equation}
						s v = s^2 = (4, 8i) \neq (8, 4i) = s t = t v.
					\end{equation}
					These computations show that Implication \eqref{eq:weakly_zero_sum_free} fails and thus, by \Cref{prop:semiring_valuations_causality}, that $\cat{Kl}(D_R)$ is not a causal Markov category.
					\qedhere
			\end{enumerate}
		\end{proof}
		
		\begin{question}
			Is there a characterization of when a representable Markov category is causal, analogous to \Cref{rep_pos}?
		\end{question}
	
	\subsection{Information flow almost surely}\label{sec:relativized}
	
		As a brief aside, let us turn to the notion of strict positivity introduced in \cite[Definition 13.16]{fritz2019synthetic}.
		As discussed in \Cref{fn:strict_pos}, the name strict positivity appears unfortunately chosen in retrospect.
		We thus refer to it as relative positivity in this article.
		It corresponds to  the positivity axiom (\cref{def:positivity}) with both antecedent and consequent relativized to $p$-almost sure equality.
		That is, we say that $\cC$ is a \newterm{relatively positive} Markov category if for all morphisms $f$, $g$, and $p$ of suitable types, we have
		\begin{equation}\label{eq:strictpositivity}
			\input{strictpositivity.tikz}
		\end{equation}

		\begin{remark}
			\label{rem:relpos_pos}
			Relative positivity implies ordinary positivity by choosing $p=\id$.  
		\end{remark}
		
		In fact, one can formulate similar relative versions of other information flow axioms, such as relative deterministic marginal independence, relative causality, etc. 
		By replacing equalities with \as{$p$} equalities in the respective proofs, implications in \cref{fig:implications} remain valid also between the relative versions of the axioms, e.g.\ relative causality implies relative positivity by \cref{prop:causal_pos}.
		
		An interesting observation is that the causality axiom is equivalent to its relativized version.
		That is, the implication
		\begin{equation}\label{causalrel}
			\input{causalrel.tikz}
		\end{equation}
		can be derived by two applications of the causality axiom itself. 
		This lets us show the following strengthening of \Cref{prop:causal_pos}.
		\begin{corollary}\label{cor:prop_causal_strictly_pos}
			If $\cC$ is causal, then $\cC$ is relatively positive.
		\end{corollary}
		\begin{proof}
			By the arguments presented in this section, we have
			\begin{equation*}
				\text{causality} \iff \text{relative causality} \implies \text{relative positivity}. \qedhere
			\end{equation*}
		\end{proof}

		It is interesting to note that the first implication cannot be generalized to quantum Markov categories due to~\cite[Example~8.28 and Proposition~8.34]{parzygnat2020inverses}.\footnote{We thank Arthur Parzygnat for pointing this out to us.}

\section{Quasi-Borel spaces, the privacy equation, and failure of positivity}
\label{sec:qbs}

	So far, counterexamples to positivity arose in settings constructed for this purpose, such as when we deal with negative probabilities. 
	The purpose of this section is to present \emph{probability theory on function spaces} as a naturally occurring situation in which positivity is violated. 
	The failure of positivity here is not a bug but a \emph{feature}, because it conforms to intuitions about information-hiding and privacy. 
	This failure of positivity is not rooted in the existence of negative probabilities like in $\finstoch_\pm$. Rather, information hiding and a form of destructive interference are central to understanding the failure of positivity in this context.

	In short, if $X \sim \nu$ is a random variable sampled from an atomless distribution $\nu$ such as a Gaussian distribution, then we can form the singleton set $A = \{X\}$, which is now a random subset of the real line. 
	As we will show, $A$ is equal in distribution to the empty set, i.e.\ its law is $\delta_\emptyset$.
	This means that the distribution of the random pair $(A,X)$ is a state which violates deterministic marginal independence. 
	Indeed its first marginal is deterministic (with value $\emptyset$), but $A$ and $X$ are \emph{not} independent; this can be seen for example because $A \ni X$ holds with probability 1.

	In order to make this counterexample precise, we first need to introduce a Markov category capable of expressing random subsets of the real line. 
	This is not possible in $\stoch$ or $\borelstoch$ because the category of (standard Borel) measurable spaces is not cartesian closed \cite{aumann1961function}. 
	\begin{itemize}
		\item In \Cref{sec:def_qbs}, we recall quasi-Borel spaces, which are a model for probabilistic programming with higher-order functions and are thus capable of formalizing our example. 
		
		\item In \Cref{sec:privacy}, we formally define the random singleton distribution as a measure on $2^\R$ and show that it equals $\delta_\emptyset$. 
			We then obtain a counterexample to deterministic marginal independence (\Cref{proposition:qbs_no_dmp}). 
			
		\item In \Cref{sec:fresh_names}, we remark that the situation in quasi-Borel spaces has strong connections to fresh name generation in computer science~\cite{sabok2021semantics}. 
			While we do not focus on the models, this matches up well with our analysis of information flow and information leaking and defines another source of interesting Markov categories.
	\end{itemize}

	\subsection{Quasi-Borel spaces}\label{sec:def_qbs}
	
		Quasi-Borel spaces have been introduced in \cite{heunen2017convenient} as a conservative extension of the category of measurable maps between standard Borel spaces to a cartesian closed category $\qbs$. 
		This means that one can form function spaces such as $2^\R$, the space of all Borel subsets of $\R$, and consider probability distributions on such objects. 
		Quasi-Borel spaces also feature a probability monad $P$ which is commutative, affine and agrees with the Giry monad on standard Borel spaces. 
		We denote the Markov category obtained as the Kleisli category of $P$ by $\qbstoch$. 
		It serves as an interesting source of counterexamples to information flow axioms, as it can ``hide'' information flow into objects like $2^\R$.
		
		A \newterm{quasi-Borel space} is a pair $(X,M_X)$, where $X$ is a set and $M_X \subseteq X^\R$ is a collection of functions $\R \to X$, called \newterm{random elements}, satisfying certain closure properties \cite{heunen2017convenient}, such as including all constant maps.
		A morphism of quasi-Borel spaces $(X,M_X) \to (Y,M_Y)$ is a function $f \colon X \to Y$ that preserves random elements. 
		We consider the following quasi-Borel spaces of interest:
		\begin{itemize}
			\item The real line is the quasi-Borel space $\R$ whose random elements are the Borel measurable maps $\R \to \R$.
			
			\item The Booleans form the quasi-Borel space $2$ with two elements, whose random elements are the Borel measurable maps $\R \to 2$, i.e.\ Borel subsets of $\R$.
			
			\item The exponential $2^\R$ in $\qbs$ consists of all Borel measurable maps $\R \to 2$.
				Its random elements $\R \to 2^\R$ are precisely the exponential transposes of Borel measurable maps of type $\R \times \R \to 2$. 
				The evaluation map 
				\begin{equation}\label{eq:evaluation}
					\mathsf{ev} \colon 2^\R \times \R \to 2, \qquad (A,x) \mapsto \llbracket x \in A \rrbracket
				\end{equation}
			is a morphism of quasi-Borel spaces, where $\llbracket x \in A \rrbracket$ stands for the truth value of the proposition $x \in A$.
		\end{itemize}
		
		Every quasi-Borel space $(X,M_X)$ has an induced $\sigma$-algebra $\Sigma_{M_X}$ given by the largest $\sigma$-algebra which makes all random elements measurable. 
		Equivalently, a subset $A \subseteq X$ is measurable if and only if its characteristic function is a morphism of quasi-Borel spaces of type $(X,M_X) \to 2$.
		
		Random elements $\R \to X$ allow one to push a source of randomness from the quasi-Borel space $\R$ onto the quasi-Borel space $(X,\Sigma_{M_X})$.
		In this spirit, a \newterm{probability measure} on $(X,M_X)$ is a probability measure on the induced measurable space $(X,\Sigma_{M_X})$ which can be obtained as a pushforward $\alpha_*\mu$, where $\alpha \in M_X$ is a random element and $\mu \in P(\R)$ is an ordinary probability measure on $\R$.
		
		\begin{example}
			The Dirac measure $\delta_x$ on $(X,\Sigma_{M_X})$ is a valid probability measure on the quasi-Borel space $(X,M_X)$, because it can be written as a pushforward of \emph{any} probability measure on $\R$ by the constant random element with image $\{x\}$.
		\end{example}
		
		The probability monad $P$ on $\qbs$ assigns, to every quasi-Borel space $(X,M_X)$, the set $P(X)$ of all probability measures on $X$ endowed with a suitable quasi-Borel structure. 
		The unit $\eta_X \colon X \to P(X)$ of the monad is given by the Dirac measure. 
		The monad $P$ is affine and commutative \cite{heunen2017convenient}, so that its Kleisli category $\qbstoch$ is a Markov category.
		
	\subsection{Random singleton sets}\label{sec:privacy}
		
		We can now take advantage of the cartesian closure of quasi-Borel spaces to define random singleton sets. 
		There is a morphism 
		\begin{equation}
			\{ \ph \} \colon \R \to 2^\R
		\end{equation}
		which sends a number $x \in \R$ to the singleton set $\{x\} \in 2^\R$. 
		Note that this morphism is simply the exponential transpose of the equality test $(=) \colon \R \times \R \to 2$. 
		If $\nu \in P(\R)$ is a probability measure on the real line, then we can describe the distribution of a \emph{random singleton set} by \[ \mathcal{RS}_\nu = P(\{ \ph \})(\nu) \in P(2^\R). \]
		Recall that by definition of distributions on a quasi-Borel space, $\mathcal{RS}_\nu$ is the pushforward measure $\{\ph\}_*\nu$ on the induced measurable space $(2^\R, \Sigma_{M_{2^\R}})$, defined by
		\begin{equation} \label{eq:def_rs}
			\mathcal{RS}_\nu(\mathcal U) = \nu \bigl( \Set{ x \in \R \given \{x\} \in \mathcal U } \bigr) \text{ for all } \mathcal U \in \Sigma_{M_{2^\R}} 
		\end{equation}
		Using the $\sigma$-algebra $\Sigma_{M_{2^\R}}$ is crucial{\,\textemdash\,}it ensures, for example, that the set $\Set{ x \in \R \given \{x\} \in \mathcal U }$ in \eqref{eq:def_rs} is measurable. 
		We elaborate on this further in the proof of the next theorem.
		
		\begin{theorem}[{Privacy equation \cite{sabok2021semantics}}]\label{thm:privacy}
			For every atomless probability measure $\nu$, the random singleton is equal in distribution to the empty set, i.e.
			\begin{equation}
				\mathcal{RS}_\nu = \delta_\emptyset.
			\end{equation}
		\end{theorem}
		\begin{proof}[Proof idea]
			The details of the proof are covered extensively in \cite[Theorem 4.1]{sabok2021semantics}. 
			We have to show that for all $\mathcal U \in \Sigma_{M_{2^\R}}$, 
			\begin{equation}\label{eq:privcymeasure}
				\mathcal{RS}_\nu(\mathcal U) = \delta_\emptyset(\mathcal U) = \llbracket \emptyset \in \mathcal U \rrbracket. 
			\end{equation}
			This claim hinges on the fact that the the induced $\sigma$-algebra on $2^\R$ is highly restrictive; we have
			\begin{equation}
				\Sigma_{M_{2^\R}} = \Set*[\Big]{ \mathcal{U} \given \forall \alpha \colon \R \times \R \to 2 \text{ measurable} , \; \Set*[\big]{ x  \given  \alpha(x, \ph) \in \mathcal{U} } \in \Sigma_\R }, 
			\end{equation}
			which is known as Borel-on-Borel in the literature on higher-order measurability (e.g.\ \cite{kechris}). 
			All families $\mathcal U$ which would be assigned different values by the formulas in \eqref{eq:privcymeasure} turn out to be \emph{not measurable}. As a brief non-example, consider the family $\mathcal E = \{ \emptyset \}$. This would clearly differentiate a random singleton from the empty set, because we have
			\begin{equation}
				\mathcal{RS}_\nu(\mathcal E) = \nu \bigl( \Set{ x \in \R  \given  \{x\} \in \mathcal E } \bigr) = 0    
					\quad \text{ and } \quad   
					\delta_\emptyset(\mathcal E) = \llbracket \emptyset \in \mathcal \mathcal E \rrbracket = 1.
			\end{equation}
			However, it can be shown that $\mathcal E \notin \Sigma_{M_{2^\R}}$. 
			In other words, checking if a set is empty is \emph{not} a morphism $2^\R \to 2$ in $\qbs$. 
		\end{proof} 
		
		We call \Cref{thm:privacy} \emph{privacy equation} because the random number $x \in \mathbb{R}$ is `anonymized in distribution' when expressed as a random set $\{x\} \in 2^\R$. 
		In particular, we have a dilation
			\begin{equation}\label{eq:joint_singleton}
				\input{joint_singleton.tikz}
			\end{equation}
		of $\nu$, with environment $2^\R$, which gives no information about the value $x$ in its environment marginal that behaves like the empty set.
		In this sense, we could view $\psi$ as a `private dilation'{\,\textemdash\,}one that leaks no information.
		
		Another way to conceive of private dilations is to require that there is no correlation between the local output and the output leaked into the environment, i.e.\ that the two outputs are independent.
		However, the dilation $\psi$ does not factorize, and so it would not be `private' in this sense.
		In particular, provided access to $x \in \R$, one can distinguish the leaked output from $\emptyset$ by applying the evaluation morphism given in \eqref{eq:evaluation}. 
		
		We employ these curious properties of $\psi$ to show that $\qbstoch$ is not a positive Markov category.
		\begin{proposition}\label{proposition:qbs_no_dmp}
			In $\qbstoch$, consider the state $\psi \colon I \to 2^\R \otimes \R$ defined by \cref{eq:joint_singleton}.
			Then the $2^\R$-marginal $\psi_{2^\R} \colon I \to 2^\R$ is deterministic, but $\psi$ is not the product of its marginals.
			Consequently, deterministic marginal independence (\Cref{def:dmp}) does not hold in $\qbstoch$.
		\end{proposition}
		\begin{proof}
			By \Cref{thm:privacy}, the $2^\R$-marginal of $\psi$ equals $\delta_\emptyset$:
			\begin{equation}
				\input{joint_singleton_marginal.tikz}
			\end{equation}
			while the second marginal equals $\nu$. 
			Therefore, the product of the marginals is $\delta_\emptyset \otimes \nu$. 
			This is different from $\psi$, as we can witness by postcomposing with the evaluation map 
			\begin{equation}
				\input{joint_singleton_ni.tikz}
			\end{equation}
			where $\mathsf{tt},\mathsf{ff} \colon 1 \to 2$ are the two boolean truth values. 
		\end{proof}
		
		By applying \Cref{thm:eqpos,prop:causal_pos}, we immediately obtain the following consequence, which had already been announced in~\cite{sabok2021semantics,stein2021structural}.
		\begin{corollary}
			$\qbstoch$ is neither a positive nor a causal Markov category.
		\end{corollary}
		Note that because $\qbstoch$ faithfully contains $\borelstoch$, positivity \emph{will} hold in the full subcategory of all quasi-Borel spaces that come from standard Borel spaces. 
		The function space $2^\R$ is not of that form. 
		This gives a novel, probabilistic reading to Aumann's result that the evaluation morphism $\mathsf{ev}$ cannot be made into a measurable map~\cite{aumann1961function}.
		It is, nevertheless, a morphism in $\qbs$.  
	
	\subsection{Fresh name generation}\label{sec:fresh_names}
	
		Fresh name generation is a classic area of computer science \cite{pitts-stark-nu,stark:cmln}. 
		A \emph{pure name} is an abstract entity which contains no other information except whether it is equal to other names. 
		Typical examples of names are identifiers such as bound variables: 
		In definitions such as $f(x,y) = xy^2$, the names of the variables $x,y$ do not matter as long as they remain distinct. 
		They could be switched or replaced by say $z,w$ without changing the meaning of the expression. 
		New names are allocated \emph{freshly}, when they are distinct from any other name already in place.
		In languages such as LISP, this primitive is called \texttt{gensym}.
		
		We give a high-level summary of the categorical semantics of name generation and show that it is another instance of information flow which can be modeled using Markov categories. 
		We also show that the information hiding for random functions arises naturally in the context of name generation, which makes it a prototypical example of non-positivity.
		
		A \emph{categorical model of name generation} consists of the following pieces of structure, satisfying further conditions spelled out in \cite[Section 4.1]{stark:cmln}:
		\begin{itemize}
			\item a cartesian closed category $\cC$ and a distinguished object $\mathbb A$ of names,
			\item an equality test $(=) \colon \mathbb A \times \mathbb A \to 2$, where $2 \coloneqq 1+1$ is assumed to exist. (One may think of the coproduct inclusions $\mathsf{tt}, \mathsf{ff} \colon 1 \to 2$ as represent the boolean truth values true and false.)
			\item a commutative affine monad $T \colon \cC \to \cC$,
			\item a distinguished state $\nu \colon I \to T(\mathbb A)$ which represents picking a fresh name.
		\end{itemize}
		Various non-degeneracy axioms are also assumed, such as that the unit $\eta_2 \colon 2 \to T(2)$ is monic.
		The Kleisli category of $T$ is a Markov category, and the freshness condition for $\nu$ states that testing a fresh name on equality always returns $\mathsf{ff}$. 
		We can write this condition in string diagrams:
		\begin{equation}\label{eq:fresh}
			\input{fresh.tikz}
		\end{equation}
		This makes $\nu$ into an abstract version of the atomless measure used in \Cref{sec:privacy}.
		
		An important problem in computer science is then to understand the behavior of higher-order functions that generate fresh names locally. 
		The program
		\begin{equation} \mathtt{let\;x=gensym()\;in\;\lambda y.(x=y)} \label{privfn} \end{equation}
		generates a fresh name $\mathtt x$ and returns a function $\mathbb A \to 2$ which tests its input $\mathtt y$ for equality with $\mathtt x$. 
		One can show that this function is observably indistinguishable from the function $\mathtt{\lambda y.false}$ \cite[Example~8]{stark:cmln}. 
		This is because the name $\mathtt{x}$ remains \emph{private} or enclosed in the function $\mathtt{\lambda y.(x=y)}$, and it can never be extracted programmatically in order to obtain the output $\mathtt{true}$. 
		
		In a categorical model of name generation, the function $\mathtt{\lambda y.(x=y)}$ is modeled using the singleton map $\{ \ph \} \colon \mathbb A \to 2^{\mathbb A}$ completely analogously to quasi-Borel spaces. 
		A model of name generation is then said to \emph{satisfy the privacy equation} if 
		\begin{equation}\label{eq:name_privacy}
			\input{name_privacy.tikz}
		\end{equation}
		holds.
		Given these ingredients, we can follow the reasoning of \cref{proposition:qbs_no_dmp} in an analogous way.
		This is a striking example how the abstraction of Markov categories enables connections between different areas of mathematics and computer science.
		That is, we obtain the following statement.
		\begin{proposition}\label{prop:name_generation}
			Every categorical model of name generation defines a Markov category. 
			If it is nondegenerate and satisfies the privacy equation, then the Markov category is not positive. 
		\end{proposition}
		It is relatively involved to construct such a model, but Stark provides one in \cite[Section 6.1]{stark:cmln} using a categorified version of logical relations. 
		We stress that this model has no probabilistic ingredients whatsoever.
		
		On the other hand, the use of Markov categories lets us apply synthetic probabilistic terminology and intuition to reason about name generation. In fact, choosing names at random is a common strategy to implement \texttt{gensym} in practice. By using an atomless measure $\nu$, we formally obtain purely probabilistic semantics for name generation, as described in \cite{sabok2021semantics}:
		
		\begin{theorem}[{\cite[Theorem~3.8]{sabok2021semantics}}]
			Quasi-Borel spaces are a nondegenerate categorical model of name generation, where $\mathbb A$ is any uncountable standard Borel space and $\nu$ any atomless measure. It furthermore satisfies the privacy equation.
		\end{theorem}

\section{Dilations and positivity properties in semicartesian categories}\label{sec:semicart}

	We now shift the focus from Markov categories to the more general semicartesian categories (\Cref{def:semicartesian}).
	A substantial theory of information flow can be developed already in semicartesian categories, as shown for example in \cite{houghtonlarsen2021dilations}. 
	This suggests that categorical probability might not need Markov categories after all, but that semicartesian categories could in fact be sufficient.
	This would be of interest not only as a conceptual clarification on the foundations of probability, but also insofar as some of its results may apply to quantum probability.

	The primary concept needed for the development of categorical probability in semicartesian terms the notion of dilation, which we have used in the previous sections in the context of Markov categories, but which is meaningful even for semicartesian categories in general.
	Throughout this section, $\cD$ thus refers to a semicartesian category.
	
	\subsection{Categories of dilations}\label{sec:dilations}
	
		In this part we study dilations more in detail (recall them from \Cref{def:dilation}).
		The following notion of dilational equality generalizes the definition of strongly almost sure equality \cite[Definition~5.7]{chojacobs2019strings}, defined originally with respect to a state $p \colon I \to X$.
		We have already encountered it in \cref{def:pes}, which defines a Markov category with parametrized equality strengthening as one in which equality almost surely implies dilational equality.
		
		\begin{definition}
			\label{def:as_eq}
			Let $p \colon A \to X$ and $f,g \colon X \to Y$ be morphisms in $\cD$. 
			We say that $f$ and $g$ are \newterm{$\bm{p}$-dilationally equal}, written as
			\[
				f \dileq{p} g,
			\]
			if for every dilation $\pi$ of $p$, we have
			\begin{equation}\label{eq:dilation_as_eq}
				\input{dilation_as_eq.tikz}
			\end{equation}
		\end{definition}
		Intuitively, $f \dileq{p} g$ means that $f$ and $g$ cannot be distinguished even with access to whatever environment that $p$ may have leaked information to.
		While this trivially implies $f p = g p$, it is typically a strictly stronger condition:

		\begin{example}
			In $\finstoch$, let $p \colon I \to \{0,1\}$ be the uniform distribution, $f = \id_{\{0,1\}}$, and $g$ the non-identity permutation on $\{0,1\}$.
			Then clearly $f p = g p$, but using $\pi \coloneqq \cop_{\{0,1\}} \circ{} p$ witnesses $f \not\dileq{p} g$.
			This is easy to understand upon noting that the distribution of a fair coin is invariant under exchanging heads and tails, but copying the outcome and switching only one copy while retaining the other clearly changes the distribution.
		\end{example}

		This example suggests that dilational equality in the Markov categories case is closely related to almost sure equality~\cite[Definition~13.1]{fritz2019synthetic}. 
		We now formalize this relation.
		
		\begin{proposition}
			\label{prop:markov_dilation}
			In a Markov category $\cC$, dilational equality implies equality almost surely.
			That is, for any $f,g \colon X \to Y$ and any $p \colon A \to X$, 
			\begin{equation}\label{eq:dileq_ase}
				f \dileq{p} g   \quad \implies \quad   f \ase{p} g .
			\end{equation}
			The converse is true if and only if $\cC$ is a causal Markov category.
		\end{proposition}
		
		\begin{proof}
			Every morphism $p$ has a dilation given by
			\begin{equation}
				\input{markov_dilation.tikz}
			\end{equation}
			Substituting it for $\pi$ in \Cref{eq:dilation_as_eq} produces exactly the desired \as{$p$} equality of $f$ and $g$.
			
			Note that the converse of implication \eqref{eq:dileq_ase} is precisely the property of parametrized equality strengthening from \cref{def:pes}, which is equivalent to causality as shown in \cref{prop:pes_causality}.
		\end{proof}
		
		To every morphism in a semicartesian category we can associate an entire category of dilations{\,\textemdash\,}a variant of the one introduced in \cite[Remark~2.2.8]{houghtonlarsen2021dilations}.
		\begin{definition}
			\label{def:cat_dilations}
			Let $p \colon A \to X$ be any morphism in $\cD$. 
			Then its \newterm{category of dilations}, denoted
			\[
				\dilations(p),
			\]
			has dilations of $p$ as objects.
			Given two dilations of $p$, say $\pi \in \cD(A, X \otimes E)$ and $\pi' \in \cD(A, X \otimes E')$, morphisms $\pi \to \pi'$ in $\dilations(p)$ correspond to $\dileq{\pi}$ equivalence classes of morphisms $f \in \cD(E, E')$ such that
			\begin{equation}\label{eq:dilation_morphism}
				\input{dilation_morphism.tikz}
			\end{equation}
			holds.
			Composition of morphisms is defined as composition of representatives in $\cD$.
		\end{definition}
		In other words, if $f_1$ and $f_2$ both satisfy \Cref{eq:dilation_morphism}, then these represent the same morphism of dilations if and only if $f_1 \dileq{\pi} f_2$, which means that for every dilation $\rho$ of $\pi$, we have
		\begin{equation}\label{eq:dilation_morphism_equivalence}
			\input{dilation_morphism_equivalence.tikz}
		\end{equation}
		where the third output wire $F$ denotes the additional environment object associated with $\rho$.
		
		\begin{lemma}
			Composition of morphisms in $\dilations(p)$ is well-defined.
		\end{lemma}
		\begin{proof}
			Consider $f_1, f_2 \in \cD(E, E')$ and $g_1, g_2 \in \cD(E',E'')$ as representatives of morphisms $\pi \to \pi'$ and $\pi' \to \pi''$.
			If $f_1 \dileq{\pi} f_2$, then we have the requisite 
			\begin{equation}
				g_1 \circ f_1 \dileq{\pi} g_1 \circ f_2
			\end{equation}
			by composing \Cref{eq:dilation_morphism_equivalence} with $g_1$.
			
			On the other hand, given $g_1 \dileq{\pi'} g_2$, we have the requisite 
			\begin{equation}
				\input{dilation_morphism_proof.tikz}
			\end{equation}
			for every dilation $\rho$ of $\pi$ because the morphism 
			\begin{equation}
				\input{dilation_morphism_proof2.tikz}
			\end{equation}
			is itself a dilation of $\pi'$. % and $g_1 \dileq{\pi'} g_2$.
			Thus, the required $g_1 \circ f_1 \dileq{\pi} g_2 \circ f_1$ follows and composition in $\dilations(p)$ is well-defined.
		\end{proof}
		
		\begin{example}[Copying leaked information is irrelevant]
			\label{ex:env_copy}
			If $\pi \colon A \to X \otimes E$ is a dilation of any morphism $p \colon A \to X$ in a causal Markov category $\cC$, then the dilation given by
			\begin{equation}
				\input{copy_dilation.tikz}
			\end{equation}
			is isomorphic to $\pi$ in $\dilations(p)$.
			To see this, note that the copy morphism itself defines a morphism of dilations $\pi \to \pi'$, while marginalizing either environment output of $\pi'$ defines a morphism $\pi' \to \pi$.
			The composite $\pi \to \pi' \to \pi$ is trivially equal to $\id_\pi$ already at the level of representatives in $\cC$.
			In the other direction, we use \Cref{prop:markov_dilation} to reduce the claim to proving $\pi'$-almost sure equality.
			This amounts to the equation
			\begin{equation}
				\input{dileq_trivial.tikz}
			\end{equation}
			which is a straightforward consequence of the commutative comonoid equations on $E$.
			Note that for this direction, the composite $\pi' \to \pi \to \pi'$ is generally not equal to $\id_{\pi'}$ on the level of representatives. In fact, in the version of $\dilations(p)$ where morphisms are \emph{not} identified up to equivalence, the two dilations are generally not isomorphic, since an isomorphism therein in particular constitutes an isomorphism between $X$ and $X \otimes X$, which typically does not exist.
		\end{example}

	\subsection{Initial dilations}\label{sec:initial_dilations}
		
		We introduce and study a variant of the concept of \emph{universal dilations} from~\cite[Definition~2.4.1]{houghtonlarsen2021dilations}.
		\begin{definition}\label{def:initial_dilation}
			An \newterm{initial dilation} of a morphism $p$ is an initial object in $\dilations(p)$.
		\end{definition}
		Explicitly, a dilation $\pi$ of $p \colon A \to X$ is initial if for every dilation $\pi'$ of $p$ there is a morphism $f$ in $\cD$ such that 
		\begin{equation}\label{eq:dilation_morphism2}
			\input{dilation_morphism.tikz}
		\end{equation}
		holds and moreover such that this $f$ is unique up to $\pi$-dilational equality.
		
		Our notion of initial dilation is intermediate between the notions of universal dilation and of complete dilation from \cite{houghtonlarsen2021dilations}.\footnote{Strictly speaking, this refers to \emph{one-sided} complete and \emph{one-sided} universal dilations, since in contrast to~\cite{houghtonlarsen2021dilations} we only consider one-sided dilations throughout (see \cref{fn:no_two_sided}).}
		That is, every universal dilation is initial and every initial dilation is complete. 
		Indeed, a universal dilation is one for which the $f$ in \cref{eq:dilation_morphism2} is unique as a morphism in $\cD$ rather than unique up to dilational equality as in the case of initial dilations.
		On the other hand, a complete dilation is one for which $f$ is merely required to exist with no uniqueness requirement.
		
		\begin{example}
			The category of finite-dimensional Hilbert spaces and quantum channels has initial dilations in the form of Stinespring dilations. 
			It is shown in \cite[Theorem 2.4.11]{houghtonlarsen2021dilations} that every \emph{minimal} Stinespring dilation is universal and thus initial{\,\textemdash\,}but in fact \emph{every} Stinespring dilation is initial.
			The argument goes as follows.
			First of all, the relevant morphism $f$ in \cref{eq:dilation_morphism2} exists because Stinespring dilations are complete \cite[Lemma 2.3.8]{houghtonlarsen2021dilations}.
			Moreover, it is unique up to dilational equality because the relation $g \dileq{\pi} h$ for any Stinespring dilation $\pi$ reduces to equality of the composites: $(\id \otimes g) \circ \pi = (\id \otimes h) \circ \pi$.
			This fact follows because every dilation of such a $\pi$ is given by a tensor product of $\pi$ with a state of the environment \cite[Corollary 2.3.23]{houghtonlarsen2021dilations}.
		\end{example}
		
		\begin{example}
			\label{pointed_sets_ex}
			Let $1 / \Setcat$ be the category of pointed sets, and consider $(1 / \Setcat)^\op$ with cartesian product as the symmetric monoidal structure.
			This is a semicartesian category.
			In the remainder of this example, we depict the arrow directions in $1 / \Setcat$ in order to avoid confusion, writing a morphism $f$ from $X$ to $Y$ as $f \colon X \leftarrow Y$.
			We denote basepoints by the symbol $\ast$.
			A dilation of a morphism $p \colon A \leftarrow X$ is a morphism $\pi \colon A \leftarrow X \times E$ such that $\pi(x,\ast) = p(x)$ for all $x \in X$.
		
			Writing $X^X$ for the hom-set $(1 / \Setcat)(X, X)$ with basepoint the identity map, function evaluation defines a morphism
			\begin{equation}
				\label{eq:pointed_sets_ev}
				\mathrm{ev}_X \colon X \longleftarrow X \times X^X,
			\end{equation}
			which is a dilation of $\id_X$.
			This dilation is initial, since every dilation $\pi \colon X \leftarrow X \times E$ of $\id_X$ arises from $\mathrm{ev}_X$ by composing with a unique morphism $X^X \leftarrow E$.
		\end{example}
		
		\begin{example}
			\label{copy_is_complete}
			In a Markov category $\cC$, a copy morphism $\cop_X$ is a dilation of $\id_X$.
			If $\cC$ is positive, then $\cop_X$ is an initial dilation of $\id_X$.
			Indeed, any dilation of $\id_X$ can be written as
			\begin{equation}\label{eq:dilation_id}
				\input{dilation_id.tikz}
			\end{equation}
			which follows from the positivity axiom in the form of deterministic marginal independence (\cref{def:dmp}), instantiated with $\pi = \iota$. 
			The uniqueness clause is automatic.
		\end{example}

		To see that this can fail without positivity, we return to the Markov category of quasi-Borel spaces from \Cref{sec:qbs}.

		\begin{proposition}\label{prop:qbs_no_dilation}
			In $\qbstoch$, the copy map $\cop_{2^\R}$ is not an initial dilation of the identity.
		\end{proposition}
		
		\begin{proof}
			In terms of the notation from \Cref{sec:qbs}, we construct a dilation $\iota \colon 2^\R \to 2^\R \otimes \R$ of the identity which cannot be obtained from the copy map. 
			To this end, we define
			\begin{equation}
				\input{joint_dilation.tikz}
			\end{equation}
			where $\cup \colon 2^\R \otimes 2^\R \to 2^\R$ takes the union of subsets. 
			That is, $\iota$ modifies its input set $A$ by adding in a random point, and records that point in the second output. 
			The map $\cup$ is a morphism of $\qbstoch$ because it can be obtained via cartesian closure and the monad unit from the disjunction map $\vee \in \qbs(2 \times 2, 2)$.
			
			The privacy equation implies that $\iota$ is indeed a dilation of the identity, because we have
			\begin{equation}
				\input{joint_dilation_pf.tikz}
			\end{equation}
			However $\iota$ cannot be obtained by some dilation morphism from the copy map, because we have
			\begin{equation}
				\input{joint_dilation_pfc.tikz}
			\end{equation}
			which can be witnessed by postcomposing with the evaluation morphism $\mathsf{ev}$ for instance.
		\end{proof}

		The copy morphism is an example of a specific type of dilation in which only the input is leaked to the environment.
		More generally, the \hypertarget{def:bloom}{\newterm{bloom}} $p_{\rm ic}$ \cite{fullwood2021informationloss} of a morphism $p \colon A \to X$ is the dilation of $p$ given by\footnote{We use the subscript ``ic'' as shorthand for ``input-copy''.}
		\begin{equation}\label{eq:bloom}
			\input{bloom.tikz}
		\end{equation}
		In the following result, we reinterpret positivity as the property that blooms of arbitrary deterministic morphisms are initial dilations.
		
		\begin{proposition}\label{prop:positivity_dilations}
			For a Markov category $\cC$, the following are equivalent:
			\begin{enumerate}[label=(\roman*)]
				\item \label{it:positivity_dilations_1} $\cC$ is positive.
				\item \label{it:positivity_dilations_2} For every deterministic $p$, its bloom $p_{\rm ic}$ is an initial dilation of $p$.
			\end{enumerate}
		\end{proposition}
		
		\begin{proof} \hfill
			\begin{itemize}
				\item[$\ref{it:positivity_dilations_1}\Rightarrow\ref{it:positivity_dilations_2}$:]
					Consider a deterministic morphism $p \colon A \to X$ and a dilation $\pi \colon A \to X \otimes E$ thereof.
					By \Cref{thm:eqpos}, we can apply deterministic marginal independence to $\pi$, which by \Cref{eq:cond_ind} gives
					\begin{equation}\label{eq:cond_ind_2}
						\input{cond_ind_2.tikz}
					\end{equation}
					This makes the bloom $p_{\mathrm{ic}}$ into an initial dilation of $p$, since the uniqueness is automatic by the fact that the $A$-marginal of $p_{\mathrm{ic}}$ is $\id_A$.
		
				\item[$\ref{it:positivity_dilations_2}\Rightarrow\ref{it:positivity_dilations_1}$:]
					We show that $\cC$ satisfies deterministic marginal independence, which is enough by \Cref{thm:eqpos}.
		
					So once again let $\pi$ be an arbitrary dilation of a deterministic morphism $p$.
					Then, by assumption, $\pi$ factors through the bloom of $p$.
					That is, there exists an $f \colon A \to E$ satisfying
					\begin{equation}\label{eq:dilation_factorization}
						\input{dilation_factorization.tikz}
					\end{equation}
					This already constitutes the relevant factorization as in \Cref{eq:cond_ind}.
					\qedhere
			\end{itemize}
		\end{proof}

		While for the existence of initial dilations for \emph{deterministic} morphisms it suffices to assume positivity, we can give a generic argument for the existence of \emph{all} initial dilations under the stronger assumption that the Markov category in question has conditionals.
		The following result and proof adapt the third author's argument for the existence of universal dilations in $\finstoch$ \cite[Theorem 2.4.6]{houghtonlarsen2021dilations}.
		
		\begin{proposition}\label{prop:conditionals_dilations}
			Let $\cC$ be a Markov category with conditionals.
			Then every morphism in $\cC$ has an initial dilation.
		\end{proposition}
		
		\begin{proof}
			Let $p \colon A \to X$ be any morphism. 
			We claim that an initial dilation of $p$ is given by the dilation which simply copies both its input and output,
			\begin{equation}\label{eq:fioc_defn}
				\input{fioc_defn.tikz}
			\end{equation}
			where the environment is given by $X \otimes A$.
			To see this, let $\pi \colon A \to X \otimes E$ be any other dilation of $p$ and let $\pi_{|X}$ be any conditional of $\pi$ with respect to $X$.
			Then $\pi$ factorizes through $p_{\mathrm{ioc}}$:
			\begin{equation}\label{eq:complete_dilation}
				\input{complete_dilation.tikz}
			\end{equation}
			This equation follows directly from the definition of conditionals and the fact that $\pi$ is a dilation of $p$.
		
			On the uniqueness, suppose that we have some other $h \colon X \otimes A \to E$ satisfying \eqref{eq:complete_dilation} in place of $\pi_{|X}$.
			Then $h$ is itself a conditional of $\pi$ with respect to $X$ and, by the \as{} uniqueness of conditionals, $h$ is $\as{p_{\rm ioc}}$ equal to $\pi_{|X}$.
			Since every Markov category with conditionals is causal \cite[Proposition 11.34]{fritz2019synthetic}, we can use \cref{prop:markov_dilation} to deduce that $h$ must also be $p_{\rm ioc}$-dilationally equal to $\pi_{|X}$, thus showing the requisite uniqueness property.
		\end{proof}
	
	\subsection{A dilational characterization of Markov categories}\label{sec:non-creative}
	
		Here, we give an abstract characterization of positive Markov categories as semicartesian categories subject to additional principles. 
		More precisely, these principles serve to single out the copy morphisms uniquely and ensure their defining properties.
		
		We first introduce the concept of \emph{non-creative} morphisms.
		It mimics the idea behind deterministic morphisms in a positive Markov category, but its definition applies in the semicartesian case since it does not reference the copy morphisms at all.
		To motivate this, let us anticipate \cref{lem:PosDetNC} below: this shows that a morphism $p$ in a positive Markov category is deterministic if and only if every dilation $\pi$ of $p$ factors as in \cref{eq:dilation_factorization}.
		While this factorization of $\pi$ makes explicit reference to $\cop_A$, we also know from \cref{copy_is_complete} that the copy morphism is an initial dilation of $\id_A$.
		Thus, the right-hand side of \cref{eq:dilation_factorization} can be expressed as a sequential composition of an arbitrary dilation 
		\begin{equation}\label{eq:id_dil}
			\input{id_dil.tikz}
		\end{equation}
		of $\id_A$ with $p \otimes \id_E$.
		Here is now the general definition.
		
		\begin{definition}%[Non-creative morphism]
			\label{def:noncreative}
			A morphism $p \colon A \to X$ in a semicartesian category is called \newterm{non-creative} if every dilation of $p$ is of the form
			\begin{equation}\label{eq:noncreative}
				\input{noncreative.tikz}
			\end{equation}
			for some dilation $\iota \colon A \to A \otimes E$ of $\id_A$.
		\end{definition}
		
		The term ``non-creative'' indicates that any information leaked from process $p$ to the environment can be viewed as having leaked already from the input of $p$.
		For a semicartesian category $\cD$, we write $\cD_\nc$ for the class of non-creative morphisms in $\cD$.
		
		\begin{lemma}[Determinism, non-creativity and positivity]
			\label{lem:PosDetNC}
			Let $\cC$ be a Markov category in which the copy morphisms are initial dilations of the identities. 
			Then every non-creative morphism is deterministic and we have $\cC_\nc = \cC_\det$ if and only if $\cC$ is positive. 
		\end{lemma}
		
		\begin{proof}
			Let $p \colon A \to X$ be a non-creative morphism in $\cC$. 
			The morphism  $\textup{copy}_X \circ p$ is a dilation of $p$, and therefore satisfies
			\begin{equation}\label{eq:outputcopy_inputdilation}
				\input{outputcopy_inputdilation.tikz}
			\end{equation}
			for some dilation $\iota \colon A \to A \times X$ of $\id_A$ by \Cref{def:noncreative}. 
			Since $\textup{copy}_A$ is an initial dilation of $\id_A$ by assumption, we must have $\iota = (\id_A \otimes f) \circ \textup{copy}_A$ for some $f \colon A \to X$. 
			Consequently, 
			\begin{equation}\label{eq:pdnc1}
				\input{pdnc1.tikz}
			\end{equation}
			holds and in particular we must have $f = p$ by marginalizing the left output.
			The resulting equation precisely asserts that $p$ is deterministic. 
			Hence the claim $\cC_\nc \subseteq \cC_\det$ follows.
		
			It remains to prove that the reverse inclusion $\cC_\det \subseteq \cC_\nc$ is equivalent to positivity of $\cC$.
			This can be achieved either by showing that $\cC_\det \subseteq \cC_\nc$ is a just a different way to express deterministic marginal independence (\cref{def:dmp}) or by relating it to property \ref{it:positivity_dilations_2} of \cref{prop:positivity_dilations}.
			We take the latter approach.
			To this end, consider a deterministic morphism $p \colon A \to X$ in $\cC$ and an arbitrary dilation $\pi \colon A \to X \otimes E$ thereof.
			
			Suppose first that $\cC$ is positive.
			By \Cref{prop:positivity_dilations}, the bloom of $p$ is an initial dilation of $p$, so that we have
			\begin{equation}\label{eq:pdnc2}
				\input{pdnc2.tikz}
			\end{equation}
			for some $f \colon A \to E$. 
			The fact that the morphism in the dashed rectangle is a dilation of $\id_A$ shows that $p$ is non-creative, so that $\cC_\det \subseteq \cC_\nc$ follows.
		
			Conversely, suppose that $\cC_\det \subseteq \cC_\nc$ holds.
			Together with the assumption that $\cop_A$ is an initial dilation of $\id_A$, this implies the factorization of $\pi$ as in \cref{eq:pdnc2} for some $f$, which is necessarily the $E$-marginal of $\pi$ and therefore unique.
			By \Cref{prop:positivity_dilations} again, we obtain that $\cC$ is a positive Markov category.
		\end{proof}
		
		Positivity implies that copy morphisms are initial dilations of the identities (\Cref{copy_is_complete}). 
		We therefore obtain another immediate characterization of positivity.

		\begin{corollary}
			\label{cor:pos_nc}
			A Markov category $\cC$ is positive if and only if the copy morphisms are initial dilations of the identities and $\cC_\det = \cC_\nc$. 
		\end{corollary}
		
		Our next goal is to characterize certain classes of Markov categories as semicartesian categories subject to additional axioms. Specifically, these axioms shall serve as a way to reconstruct the copy morphisms and to ensure their required properties. In this way, the copy morphisms emerge as a consequence of dilational axioms rather than being imposed as additional mathematical data on top of a semicartesian category.
		\begin{definition}\label{def:broadcasting}
			\begin{enumerate}
				\item A morphism $b_X \colon X \to X \otimes X$ is \newterm{broadcasting} if both of its marginals are the identity:
					\begin{equation}\label{eq:double_dil}
						\input{double_dil.tikz}
					\end{equation}
				\item An object $X$ \newterm{admits broadcasting} if there is a broadcasting morphism $b_X \colon X \to X \otimes X$.
			\end{enumerate}
		\end{definition}
		The terminology is inspired by the analogous concept of \emph{broadcasting} in quantum information theory \cite{barnum1996broadcast}.
		For example, the terminal object $I$ trivially admits broadcasting; and if $X$ and $Y$ admit broadcasting via $b_X$ and $b_Y$, then so does $X \otimes Y$ via the product $b_X \otimes b_Y$ with the middle outputs swapped.
		Furthermore, every copy morphism in a Markov category is broadcasting by the counitality axiom.
		In some Markov categories, however, there are other broadcasting morphisms as well.
		For instance, in $\finstoch_\pm$ (\cref{pm_fails_dmp}) for ${X = \{0,1\}}$, the kernel given by
		\begin{equation}
			b_X( x_1, x_2 | x_0) \coloneqq (-1)^{x_1+x_2} + \cop_X(x_1, x_2 | x_0)
		\end{equation}
		is broadcasting, but not equal to the copy map.

		On the positive side, in positive Markov categories there are no broadcasting morphisms other than copy,\footnotemark{} since then for any broadcasting morphism $b_X$, we have
		\footnotetext{This is a special case of a general result on broadcasting morphisms in quantum Markov categories \cite[Theorem 4.17]{parzygnat2020inverses} based on ``quantum positivity'' (see also \cref{rem:positive} \ref{item:quantumpositive}).}%
		\begin{equation}
			\input{positive_broadcasting.tikz}
		\end{equation}
		where the second step is by positivity and the third by the broadcasting property.
		This shows that $b_X$ is the copy morphism upon marginalizing the third output.
			
	\begin{proposition}\label{prop:nocloning}
		Let $\cD$ be a semicartesian category. For any object $X$, the following are equivalent:
		\begin{enumerate}
			\item \label{cloning} $X$ admits broadcasting.
			\item \label{noncreative_deletion} The discard morphism $\discard_X$ is non-creative.
		\end{enumerate}
	\end{proposition}
	We therefore obtain a \emph{no-broadcasting theorem} \cite{barnum2007nobroadcasting}: $X$ does not admit broadcasting if and only if $\discard_X$ is not non-creative.
	\begin{proof}
		\hfill
		\begin{itemize}
			\item [$\ref{cloning}\Rightarrow\ref{noncreative_deletion}$:] 
				Given a broadcasting morphism $b_X$, by equations \eqref{eq:double_dil} we have 
				\begin{equation}
					\input{noncreative_2.tikz}
				\end{equation}
				for every $f \colon X \to Y$.
				In particular, an arbitrary dilation of $\discard_X$ is of the form of diagram~\eqref{eq:noncreative} and thus $\discard_X$ is non-creative.

			\item[$\ref{noncreative_deletion}\Rightarrow\ref{cloning}$:]
				Consider the dilation $\id_X$ of $\discard_X$ with environment $X$. 
				Since $\discard_X$ is non-creative, this dilation must be of the form \eqref{eq:noncreative}.
				That is, there is a dilation $\iota \colon X \to X \otimes X$ of $\id_X$ whose second output is the environment and which also satisfies
				\begin{equation}
					\input{del_noncreative.tikz}
				\end{equation}
				In particular, $\iota$ is broadcasting.
				 \qedhere
		\end{itemize}
	\end{proof}

		In the remainder of this section, we characterize positive Markov categories in purely semicartesian terms.
		
		\begin{theorem}
			\label{thm:acp}
			Let $\cD$ be a semicartesian category. 
			Then the following are equivalent:
			\begin{enumerate}
				\item\label{is_markov_positive} $\cD$ can be made into a Markov category in which the copy morphisms are initial dilations of the identities.
				\item\label{semicartesian_stuff} For every object $X$, the identity $\id_X$ admits an initial dilation $\iota \colon X \to X \otimes E$ such that the marginal
				\begin{equation}
					\input{marg_dil.tikz}
				\end{equation}
				is non-creative.
			\end{enumerate}
			If these conditions hold, then the Markov category structure in \ref{is_markov_positive} is unique.
		\end{theorem}

		%In particular, we don't need to also add the corresponding version of monoidal multiplicativity condition (see \Cref{copy_multiplicative}) relating $c_{A}$ and $c_{X}$ to $c_{A \otimes X}$.
		%Furthermore, both conditions in item \ref{semicartesian_stuff} are mere properties of $\cD$ rather than extra structure, which corresponds to the fact that the copy morphisms in a positive Markov category are unique~\cite[Remark~11.29]{fritz2019synthetic}.
		%Notice also that, in condition \ref{semicartesian_stuff}, for each pair of objects $A$ and $X$, the dilations $c_A$ and $c_X$ need to be compatible via the condition involving \eqref{eq:acp3}.

		\begin{proof} \hfill
			\begin{itemize}
				\item[$\ref{is_markov_positive}\Rightarrow\ref{semicartesian_stuff}$:]
					If $\cD$ is a Markov category in which the copy morphisms are initial dilations of the identity, then the requirements of \ref{semicartesian_stuff} hold since the identity is non-creative.
		
				\item[$\ref{semicartesian_stuff}\Rightarrow\ref{is_markov_positive}$:]
					Using the notation of the statement, since $\iota$ is a dilation of the non-creative morphism $\iota_E$ with environment $X$, we have 
					\begin{equation}\label{eq:acp8}
						\input{acp8.tikz}
					\end{equation}
					for some $c_X\colon X \to X \otimes X$ whose right-hand marginal is $\id_X$.
					We argue that this morphism is broadcasting.
					Indeed, we just noted that the right-hand marginal is the identity, while the left-hand marginal is the identity because $\iota$ is a dilation of the identity:
					\begin{equation}
						\input{acp10.tikz}
					\end{equation}
					Moreover, since $\iota$ is an initial dilation of the identity, there exists a $g \colon E \to X$ in $\cD$ that corresponds to a morphism of type $\iota \to c_X$ in $\dilations(\id_X)$, so that we have
					\begin{equation}
						\input{acp9.tikz}
					\end{equation}
					Marginalizing the first output of this equation gives $g \circ \iota_E = \id_X$, while the initiality of $\iota$ implies ${\iota_E \circ g \dileq{\iota} \id_{E}}$.
					As a consequence, $c_X$ is isomorphic to $\iota$ in $\dilations(\id_X)$ and thus also an initial dilation of $\id_X$. 

					Next, we show that every identity morphism $\id_X$ has exactly one broadcasting morphism. 
					To this end, let ${c'_X \colon X \to X \otimes X}$ be another morphism satisfying equations \eqref{eq:double_dil}.
					Then we have
					\begin{equation}\label{eq:acp4}
						\input{acp4.tikz}
					\end{equation}
					for some $h \colon X \to X$ by initiality of $c_X$. 
					Marginalizing the left output gives $h = \id_X$, so that $c'_X$ is necessarily equal to $c_X$. 
					This demonstrates the asserted uniqueness.  
					In particular, this immediately implies that $c_X$ is symmetric, i.e.\ it satisfies
					\begin{equation} \label{eq:swap}
						\input{acp2.tikz}
					\end{equation}
					because swapping the outputs of $c_X$ obviously results in a broadcasting morphism again.
					Furthermore, $c_X$ is also the only dilation of $\id_X$ which is symmetric in this sense, as being a symmetric dilation implies being broadcasting.

					Next, let us show that the unique symmetric dilation equips $X$ with the structure of a commutative comonoid.
					Commutativity is precisely \cref{eq:swap}, and counitality corresponds to equations \eqref{eq:double_dil}.
					For coassociativity, the initiality of $c_X$ implies that there is a $d_X$ of the same type such that 
					\begin{equation}\label{eq:acp6}
						\input{acp6.tikz}
					\end{equation}
					holds, since the left-hand side of the equation is a dilation of $\id_X$ with environment $X \otimes X$.
					But then marginalizing the first output shows $d_X = c_X$ since $c_X$ is broadcasting, so that we obtain the required coassociativity equation.
		
					In order to see that $\cD$ becomes a Markov category, it is thus enough to prove the multiplicativity equation
					\begin{equation}\label{eq:acp7}
						\input{acp7.tikz}
					\end{equation}
					for all objects $X$ and $Y$.
					As we argued above, symmetric dilations of identities are unique.
					Therefore, \cref{eq:acp7} follows upon showing that its right-hand side is a dilation of $\id_{X \otimes Y}$ that is invariant under swapping the two copies of $X \otimes Y$.
					Since this is a direct consequence of the symmetry of $c_X$ and $c_Y$, we obtain the desired result.

					Furthermore, the uniqueness of the broadcasting morphisms implies that the constructed Markov category structure is the only possible one (cf.\ \cite[Remark 11.29]{fritz2019synthetic}).
					%Now that the Markov category structure has been constructed, it remains to show that it is positive.
					%But this is now a consequence of \Cref{lem:PosDetNC}.
					\qedhere
			\end{itemize}
		\end{proof}
		
		With this characterisation at hand, also positive Markov categories can now be characterized in semicartesian terms by adding additional requirements to item~\ref{semicartesian_stuff}.

		\begin{corollary}[Semicartesian characterization of positive Markov categories] 
			\label{cor:acp}
			Let $\cD$ be a semicartesian category. 
			Then the following are equivalent:
			\begin{enumerate}
				\item \label{is_markov_positive1} $\cD$ can be equipped with copy morphisms making it into a positive Markov category.
				\item \label{semicartesian_stuff1} For every object $X$, the identity $\id_X$ admits an initial dilation $\iota \colon X \to X \otimes E$ such that the marginal
				\begin{equation}
					\input{marg_dil.tikz}
				\end{equation}
				is non-creative. 
				Moreover, if $f \colon X \to Y$ is such that for any dilation $\pi$ of the identity morphism $\id_Y$ we have
				\begin{equation}\label{eq:det_pullback}
					\input{det_pullback.tikz}
				\end{equation}
				for some dilation $\pi'$ of the identity morphism $\id_X$, then $f$ is non-creative.
			\end{enumerate}
		\end{corollary}
		\begin{proof} 
			By \Cref{thm:acp}, we can assume that $\cD$ is a Markov category in which the copy morphisms are initial dilations of the identity. 
			The problem is therefore reduced to showing that, under this assumption, positivity is equivalent to the second part of item \ref{semicartesian_stuff1}. 
				
			Consider a morphism $f$ satisfying the condition of \cref{eq:det_pullback}. 
			For $\pi=\cop_Y$, we have 
			\begin{equation}\label{eq:det_pullback_copy}
				\input{det_pullback_copy.tikz}
			\end{equation}
			where the box on the right-hand side represents an arbitrary dilation of the identity on $X$ since $\cop_X$ is its initial dilation by assumption. 
			Marginalizing the left output gives $g = f$, so that $f$ is deterministic by \cref{eq:det_pullback_copy}.
			Conversely, every deterministic morphism satisfies \cref{eq:det_pullback}, which, once again, follows from copy morphisms being initial dilations of identities. 
			Therefore, the second part of item \ref{semicartesian_stuff1} can be restated as ``every deterministic morphism is non-creative'', and this holds if and only if $\cD$ is positive by \cref{lem:PosDetNC}.
	\end{proof}
		\begin{remark}
			\label{semicar_markov_comments}
			Let us make a couple of comments on \cref{thm:acp,cor:acp}.
			\begin{enumerate}
				\item The conditions stated in items \ref{semicartesian_stuff} of both results suggests that being a Markov category of the specified sort is a mere \emph{property} of a symmetric monoidal category rather than extra structure.

					If we restrict to strict monoidal structure for simplicity, then this statement is easy to make precise in terms of the formalism of \emph{stuff, structure and property}~\cite{baez2010lectures}: the obvious forgetful functor from the category of strict Markov categories and Markov functors~\cite[Definition~10.14]{fritz2019synthetic} to the category of strict symmetric monoidal categories and strict monoidal functors is full and faithful, and therefore forgets at most property.
					While the faithfulness is trivial, the fullness amounts to the statement that every strict monoidal functor between positive Markov categories preserves the copy morphisms.
					This holds because such a functor clearly maps a copy morphism to a broadcasting morphism, which in the positive case must be a copy morphism again.
				\item \Cref{thm:acp} also holds substituting all the occurrences of ``initial dilation'' with different properties{\,\textemdash\,}namely that of ``complete dilation'' \cite[Definition 2.3.1]{houghtonlarsen2021dilations} and ``universal dilation'' \cite[Definition 2.4.1]{houghtonlarsen2021dilations}.
				%when we replace the requirement that $\pi$ is an initial dilation with a weaker property{\,\textemdash\,}namely that it is a complete dilation \cite[Definition 2.3.1]{houghtonlarsen2021dilations} of $\id_X$.
					A complete dilation is like an initial dilation, but for the fact that there could be multiple morphisms $f$ in $\dilations(p)$ that relate it to another dilation of $p$ via \cref{eq:dilation_morphism}, while a universal dilation is an initial dilation for which the morphism $f$ is unique as a morphism in $\cD$.
				\item \Cref{pointed_sets_ex} on pointed sets shows that there are semicartesian categories which have initial dilations, but not initial dilations satisfying item \ref{semicartesian_stuff} of \cref{thm:acp}.
					Indeed, we have
					\begin{equation}\label{eq:eval_marginal}
						\input{eval_marginal.tikz}
					\end{equation}
					i.e.\ the marginal of the initial dilation given in \eqref{eq:pointed_sets_ev} (the left-hand side) equals the morphism $X \leftarrow X^X$ that evaluates functions at the basepoint.
					Since the functions under consideration are all basepoint-preserving, this is indeed the map that sends every function in $X^X$ to the basepoint $\ast$ of $X$, i.e.\ the right-hand side of \cref{eq:eval_marginal}, where $\ast \colon I \leftarrow X^X$ is the unique morphism of this type.
					If the constant morphism from \cref{eq:eval_marginal} was non-creative, then there would have to be a factorization of the form
					\begin{equation}
						\input{pointed_sets_ev_noncreative.tikz}
					\end{equation}
					for some dilation $\iota$ of $\id_X$.
					For any $X$ with at least two elements, this equation cannot be satisfied, because after plugging in a function $f \in X^X$, its right-hand side is independent of $f$ while the left-hand side is not.
%					such as $X = $, this clearly contradicts the fact that the marginal on the right-hand side is constant.
					
				\item Another important example is the category of finite-dimensional Hilbert spaces and quantum channels.
					In this category, identities only have trivial dilations (see e.g.\ \cite[Section 2.2.A]{houghtonlarsen2021dilations}), which is a strong form of the no-cloning theorem (cf.\ \cite[Proposition 7.1]{selbyleak}). 
					In particular, there is no copy morphism and so also item \ref{semicartesian_stuff} of \cref{thm:acp} cannot hold. 
					The reason why becomes clearer in light of \cref{prop:nocloning}. 
%					This implies that the delete morphisms are not non-creative. Conversely, notice that the delete morphism being not non-creative actually tells us that there is no copy, since the identity is a dilation for the delete morphism. 
%					Furthermore, item \ref{semicartesian_stuff} of \cref{thm:acp} gives 
%					Hence every identity morphism itself is its own initial dilation, and \Cref{eq:swap} once again fails.
			\end{enumerate}
		\end{remark}

\newpage
\bibliographystyle{plain}
\bibliography{markov}

\end{document}